\documentclass{amsart}
\usepackage[utf8]{inputenc}
\usepackage{hyperref}
\usepackage{amsfonts}
\usepackage{amsmath}
\usepackage{amssymb}
\usepackage{amscd}
\usepackage{graphicx}
\usepackage{enumitem}
\usepackage{latexsym}
\usepackage{color}
\usepackage{comment}
\usepackage{epsfig}
\usepackage{amsopn}
\usepackage{xypic}
\usepackage{mathrsfs}
\usepackage{array}
\usepackage[usenames,dvipsnames]{pstricks}
\usepackage{epsfig}
\usepackage{pst-grad} 
\usepackage{pst-plot} 
\usepackage[space]{grffile} 
\usepackage{etoolbox} 
\makeatletter 
\patchcmd\Gread@eps{\@inputcheck#1 }{\@inputcheck"#1"\relax}{}{}
\makeatother

\usepackage{float}
\usepackage{tikz-cd}
\usepackage[
backend=biber,
style=alphabetic,
sorting=nyt
]{biblatex}
\hypersetup{pdfpagemode=UseNone}
\usepackage{booktabs}
\usepackage{longtable}
\newtheorem{lemma}{Lemma}[section]

\newtheorem{theorem}[lemma]{Theorem}

\newtheorem{corollary}[lemma]{Corollary}
\newtheorem{proposition}[lemma]{Proposition}

\newtheorem{example}[lemma]{Example}

\theoremstyle{definition}
\newtheorem{examplenotitalic}[lemma]{Example}

\newenvironment{proofofthm4.3}{\textit{Proof of Theorem 4.3.}\/}{\hfill$\Box$}

\newcommand{\cI}{\mathcal{I}}

\newcommand{\cO}{\mathcal{O}}

\renewcommand{\P}{\mathbb{P}}

\DeclareMathOperator{\Aut}{Aut}

\DeclareMathOperator{\Hom}{Hom}

\DeclareMathOperator{\rk}{rk}

\DeclareMathOperator{\sHom}{\mathcal{H}\kern -.5pt\mathit{om}}
\DeclareMathOperator{\sTor}{\mathcal{T}\kern -1.5pt\mathit{or}}

\DeclareMathOperator{\codim}{codim}

\DeclareMathOperator{\sEnd}{\mathcal{E}\kern -.5pt\mathit{nd}}
\DeclareMathOperator{\sExt}{\mathcal{E}\kern -.5pt\mathit{xt}}

\title{Normal Bundles of Rational Normal Curves on Hypersurfaces}
\author{Lucas Mioranci}
\address{Department of Mathematics, Statistics, and Computer Science, University of Illinois at Chicago, Chicago, IL}
\email{lmiora2@uic.edu}
\date{}

\begin{document}

\maketitle

\begin{abstract}

    Let $C$ be the rational normal curve of degree $e$ in $\P^n$, and let $X\subset \P^n$ be a degree $d\ge 2$ hypersurface containing $C$. I. Coskun and E. Riedl show that the normal bundle $N_{C/X}$ is balanced for a general $X$. H. Larson studies the case of lines ($e=1$) and computes the dimension of the space of hypersurfaces for which $N_{C/X}$ has a given splitting type. In this paper, we work with any $e\ge 2$. We compute explicit examples of hypersurfaces for all possible splitting types, and for $d\ge 3$, we compute the dimension of the space of hypersurfaces for which $N_{C/X}$ has a given splitting type. For $d=2$, we give a lower bound on the maximum rank of quadrics $X$ such that $N_{C/X}$ has a fixed splitting type.
    
\end{abstract}

\section{Introduction}

Rational curves play an important role in the study of birational and arithmetic geometry of projective varieties. The local structure of the space of rational curves on a variety is determined by its normal bundle. In this paper, we study the possible splitting types of the normal bundle of rational normal curves on a hypersurface in $\P^n$ and obtain explicit examples of hypersurfaces for each splitting type. Unless otherwise specified, we work over an algebraically closed field $K$ of arbitrary characteristic $p$.

Let $X$ be a degree $d$ hypersurface in $\P^n$ containing a smooth rational curve $C$. By the Birkhoff-Grothendieck theorem, the normal bundle of $C$ on $X$, $N_{C/X}$, splits as a direct sum $N_{C/X}\cong \bigoplus_{i=1}^{n-2}\cO_{\P^1}(a_i)$. The collection of integers $a_i$ is called the \textit{splitting type} of $N_{C/X}$. We say a splitting type is \textit{balanced} when $|a_i-a_j|\le 1$ for all $i$ and $j$. We are interested in studying the possible splitting types when $C$ is the rational normal curve of degree $e$ in $\P^n$. Additionally, given a splitting type $\vec{a}$, let $E_{\vec{a}}=\bigoplus_{i=1}^{n-2}\cO(a_i)$. We study the space of hypersurfaces $X$ containing $C$ and for which $N_{C/X}\cong E_{\vec{a}}$. Let $X$ be given by a degree $d$ polynomial $F$ as $X = V(F)$, and define
\[ \Sigma = \{ F \ | \ X \text{ is a degree } d \text{ hypersurface smooth along } C \}\subset H^0(\cO_{\P^n}(d)) \]
and
\[ \Sigma_{\vec{a}} = \{ F\in \Sigma \ | \ N_{C/X}\cong E_{\vec{a}} \}\subset \Sigma . \]
The codimension of the locus of vector bundles on $\P^1$ with a specified splitting type $E_{\vec{a}}$ in the versal deformation space is given by $h^1(\P^1, \sEnd(E_{\vec{a}}))$ (see \cite{C08} Lemma 2.4 and \cite{CR18}). We call it the \textit{expected codimension} for $\Sigma_{\vec{a}}$ in $\Sigma$ and observe that
\[ h^1(\sEnd(E_{\vec{a}})) = h^1(E_{\vec{a}}^*\otimes E_{\vec{a}}) = \sum_{\{i,j | a_i-a_j\le -2\}}(a_j-a_i-1). \]

Normal bundles of rational curves have been studied in many different works (see \cite{Sa80}, \cite{GS80}, \cite{EV81}, \cite{EV82}, \cite{M86}, \cite{Ran}, \cite{AlzatiRe}, \cite{AlzatiReTortora}, \cite{CR18}, \cite{Ran20}, \cite{Ran22}). In \cite{CR18}, I. Coskun and E. Riedl show that the locus of rational curves $C$ in $\P^n$ whose normal bundle $N_{C/\P^n}$ has a given splitting type has arbitrarily many components, and the difference between the expected dimension and actual dimension of a component can grow arbitrarily large as the degree of the curve increases.

For rational curves on hypersurfaces, H. Larson (\cite{L21}, proof of Theorem 1.1) studies the case of lines and shows that $\Sigma_{\vec{a}}$ is smooth of the expected codimension. For rational normal curves of degree $e\ge 2$, I. Coskun and E. Riedl prove (\cite{CR} Corollary 3.8) that the normal bundle on a degree $d\ge 2$ general hypersurface is balanced. Here, we examine more closely the case of rational normal curves on hypersurfaces, and find explicit examples of hypersurfaces $X$ for each possible splitting type of $N_{C/X}$.

In section 2, we describe the normal bundle as the kernel in the sequence
\[ 0\longrightarrow N_{C/X}\longrightarrow \cO(e+2)^{e-1}\oplus \cO(e)^{n-e}\overset{\psi_F}{\longrightarrow} \cO(de)\longrightarrow 0, \]
so the splitting type of $N_{C/X}$ must have the form of a rank $(n-2)$ direct sum
\[ N_{C/X}\cong \left(\bigoplus_{i=1}^{e-1}\cO(e+2-a_i)\right )\oplus \left (\bigoplus_{j=e+2}^{n}\cO(e-b_j) \right ), \]
with $a_i, b_j\ge 0$, and degree $\deg N_{C/X} = \deg N_{C/\P^n} - \deg \cO(de) = e(n-d+1)-2$, equivalently $\sum_{i=1}^{e-1}a_i + \sum_{j=e+2}^n b_j = e(d-1)$. We ask which of these splitting types are achieved as the normal bundle of a hypersurface $X$ smooth along $C$. In addition, we find explicit examples of $X$ for each possible splitting type. We first examine the case $d\ge 3$:

\begin{theorem}\label{existencetheorem} (cf. Theorem \ref{existence})
Let $e\le n$, $d\ge 3$, and let $C$ be the rational normal curve of degree $e$ in $\P^n$. Then:
\leavevmode
\begin{enumerate}[label=(\alph*)]
    \item For all splitting types with at most $e-2$ summands of degrees greater than $e$, that is, of the form
    \[ E = \left(\bigoplus_{i=1}^{e-2}\cO(e+2-a_i)\right )\oplus \left (\bigoplus_{j=e+1}^{n}\cO(e-b_j) \right ), \]
    with $a_i, b_j\ge 0$ and $\sum_{i=1}^{e-2}a_i + \sum_{j=e+1}^n b_j = e(d-1)-2$, we obtain explicit  examples of degree $d$ hypersurfaces $X$, smooth along the curve $C$, with normal bundle $N_{C/X}\cong E$.  
\newline
    \item If $N_{C/X}$ is not of the form $E$ above, then $e<n-1$, and the normal bundle must have the form
    \[ E' = \cO(e+2)^{e-1}\oplus \left (\bigoplus_{j=e+2}^{n}\cO(e-b_j) \right ), \]
    with $b_j\ge 0$ and $\sum_{j=e+2}^n b_j = e(d-1)$.
    For each splitting type $E'$, we find explicit examples of degree $d$ hypersurfaces $X$, smooth along $C$, with normal bundle $N_{C/X}\cong E'$.
\end{enumerate}
\end{theorem}

When $d\ge 4$, Bertini's Theorem implies that the general example is smooth.
\begin{corollary}\label{existsmooth} (cf. Corollary \ref{smoothexample})
(char $K=0$) Let $d\ge 4$, and assume the base field $K$ has characteristic $0$. For all splitting types $E$ and $E'$ as in the theorem, there exists a smooth hypersurface $X$ of degree $d$ containing the curve $C$ with normal bundle $N_{C/X}\cong E$.
\end{corollary}

In addition, for $d\ge 3$, we compute the dimension of $\Sigma_{\vec{a}}$ and show that, when $e=n$ or the splitting type does not have terms of degree $e+2$, we get the expected codimension in $\Sigma$. When the dimension is not the expected one, we can compute the difference to its actual dimension. In particular, we show the difference can get arbitrarily large as $n$ grows.

\begin{theorem} (cf. Theorem \ref{teoremasigmaa})
Let $e\le n$ and $d\ge 3$. Given a splitting type $E_{\vec{a}}$ as in Theorem \ref{existencetheorem}, the locus $\Sigma_{\vec{a}}$ is irreducible and smooth of codimension $h^1(\sEnd(E_{\vec{a}})) - h^1(\sHom(E_{\vec{a}}, N_{C/\P^n}))$ in $\Sigma$.
\end{theorem}

\begin{corollary} (cf. Corollary \ref{difference}) Let $z = |\{ i \ | \ a_i = e+2 \}|$ be the number of terms $\cO(e+2)$ in the splitting type $E_{\vec{a}}$. Then the codimension of $\Sigma_{\vec{a}}$ in $\Sigma$ is $h^1(\sEnd(E_{\vec{a}})) - (n-e)z$.
\end{corollary}

The case $d=2$ presents additional difficulties, as for example, the map $\phi : H^0(\cI_C(2))\to \Hom (N_{C/\P^n}, \cO_{\P^1}(2e))$ is not surjective, so it is not clear which maps $\psi_F$ are induced by hypersurfaces $X$. Nonetheless, we can adapt the computation from the case $d\ge 3$, and find explicit examples of quadrics for every possible splitting type.

\begin{theorem}\label{quadricsexamples} (cf. Theorem \ref{quadricprop1})
Let $e\le n$, $d=2$, and let $C$ be the rational normal curve of degree $e$ in $\P^n$. Then:
\leavevmode
\begin{enumerate}[label=(\alph*)]
    \item For any given splitting type of the form
    \[ E = \left ( \bigoplus_{i=1}^{e-2}\cO(e+2-a_i) \right )\oplus \left ( \bigoplus _{j=e+1}^{n}\cO(e-b_j)\right ), \] with $a_i, b_j\ge 0$ and $\sum_{i=1}^{e-2}a_i + \sum_{j=e+1}^nb_j = e-2$, we produce a quadric $X=V(F)$, smooth along $C$, for which $N_{C/X}\cong E$.

    More explicitly, rearrange the $a_i$ in non-decreasing order $0 = a_y < a_{y+1}\le \cdots \le a_{e-2}$, and let $\beta_0 = 0$ and $\beta_i = a_{y+1} + \cdots + a_{y+i}$ for $i=1, \ldots , e-2-y$:
    
    For $e=n$, the quadric $X$ given by the polynomial
    \[ F = \sum_{i=0}^{n-2-y}Q_{\beta_i+1,\beta_i+2} =  Q_{1,2} + Q_{\beta_1 + 1, \beta_1 +2} + Q_{\beta_2 + 1, \beta_2 + 2} +  \cdots + Q_{\beta_{n-3-y}+1, \beta_{n-3-y}+2} + Q_{n-1,n}, \]
    where $Q_{ij} = x_ix_{j-1}-x_{i-1}x_j$, has normal bundle $N_{C/X}\cong E$.
    
    For $e<n$, let also $\gamma_n = e$ and $\gamma_{j} = e - b_n - b_{n-1} - \cdots - b_{j+1}$ for $e+1\le j\le n-1$. Then a quadric $X$ such that $N_{C/X}\cong E$ is given by
    \begin{align*}
        F = & \sum_{i=0}^{e-2-y}Q_{\beta_i+1,\beta_i+2} + \sum_{j=e+1}^n x_{\gamma_j}x_j\\
        = & (Q_{1,2} + Q_{\beta_1 + 1, \beta_1 +2} + \cdots + Q_{\beta_{e-2-y}+1, \beta_{e-2-y}+2}) + (x_{\gamma_{e+1}}x_{e+1} + \cdots + x_{\gamma_{n-1}}x_{n-1} + x_ex_n).
    \end{align*}
    
    \item If the splitting type of $N_{C/X}$ contains $e-1$ terms of degree greater than $e$, then it must be of the form
    \[ E' = \cO(e+2)^{e-1}\oplus \left ( \bigoplus _{j=e+2}^{n}\cO(e-b_j)\right ), \]
    with $b_j\ge 0$, $\sum_{j=e+2}^n b_j = e$ and $e<n-1$. In this case, $X$ contains the $e$-plane spanned by $C$. We produce a quadric $X$, smooth along $C$, for which $N_{C/X}\cong E'$.
    
    More explicity, let $\gamma_n = e$ and $\gamma_j = e-b_n-b_{n-1}-\cdots -b_{j+1}$ for $e+1\le j\le n-1$, and let $X$ be the quadric given by the polynomial
    \[ F = \sum_{j=e+1}^n x_{\gamma_j}x_j = x_0x_{e+1} + x_{b_{e+2}}x_{e+2} + \cdots + x_{\gamma_{n-1}}x_{n-1} + x_ex_n. \]
    Then $N_{C/X}\cong E'$.

\end{enumerate}
\end{theorem}

The quadrics constructed in Theorem \ref{quadricsexamples} are smooth along the curve $C$ but not necessarily smooth quadrics. We can, however, work on their quadratic form matrices to produce examples with a smaller singular locus. This allows us to find a lower bound for the maximum rank of quadrics with a given normal bundle:

\begin{theorem}\label{quadricstheorem} (cf. Theorem \ref{corankexample})
\leavevmode
\begin{enumerate}[label=(\alph*)]
    \item For every splitting type
    \[ E = \left ( \bigoplus_{i=1}^{e-2}\cO(e+2-a_i) \right )\oplus \left ( \bigoplus _{j=e+1}^{n}\cO(e-b_j)\right ), \]
    with $a_i, b_j\ge 0$ and $\sum_{i=1}^{e-2}a_i + \sum_{j=e+1}^nb_j = e-2$, we obtain an example of quadric $X$ of corank at most $\sum _{a_i\ge 4}(a_i-3)$ with $N_{C/X}\cong E$. In particular, if $a_i\le 3$ for all $i$, there exists a smooth quadric $X$ with $N_{C/X}\cong E$.
    \newline
    \item For splitting types of the form
    \[ E' = \cO(e+2)^{e-1}\oplus \left ( \bigoplus _{j=e+2}^{n}\cO(e-b_j)\right ), \]
    with $b_j\ge 0$ and $\sum _{j=e+2}^n b_j = e$, let $w = |\{ j \ | \ b_j = 0 \}|$ be the number of terms of degree $e$. Then we obtain a quadric $X$ of corank $e+1 - \mathrm{min}\{e+1, n-e-w\}$ with $N_{C/X}\cong E'$. In particular, if $e+1\le n-e-w$, then there exists a smooth quadric $X$ with $N_{C/X}\cong E'$.
\end{enumerate}
\end{theorem}

\noindent \textbf{Organization of the paper.} In Section 2, we discuss definitions and computations that will be used throughout the paper. In Section 3, we work on the case $d\ge 3$, both providing the example of a hypersurface with $N_{C/X}\cong E_{\vec{a}}$ and computing the dimension of $\Sigma_{\vec{a}}$. Section 4 is dedicated to the case $d=2$; we compute the example with a given splitting type, and prove Theorem \ref{quadricstheorem}.

\medskip

\noindent \textbf{Acknowledgements.} I would like to thank my advisor Izzet Coskun for our many discussions and his guidance throughout the writing of this paper. I would also like to thank Benjamin Gould, Yeqin Liu, Eric Riedl, Geoffrey Smith, and Herivelto Borges for reading my drafts and giving valuable suggestions. I thank the referee for their helpful comments and for spotting a missing case in the theorem.

\section{Preliminaries}

  In this section we cover the definitions and results necessary to describe our approach. Most results here are proved in \cite{CR}. We describe the rational normal curves, and give an explicit computation of their normal bundle on a given hypersurface. We take coordinates $(x_0: \cdots : x_n)$ on $\P^n$.
  
\medskip

\subsection{Rational Normal Curves} For $e\le n$, we say the rational normal curve $R_e$ of degree $e$ in $\P^n$ is the curve defined by
\[ (s^e : s^{e-1}t : s^{e-2}t^2 : \cdots : st^{e-1} : t^e : 0 : \cdots : 0): \P^1\to \P^n. \]

Its homogeneous ideal $I_{R_e}\subset K[x_0, \ldots , x_n]$ is cut out by quadrics and linear forms:
\[ I_{R_e} = ( \{ Q_{ij} = x_ix_{j-1} - x_{i-1}x_j \mid 1\le i < j\le e \}\cup \{x_{e+1}, \ldots , x_n \} ). \]

The $Q_{ij}$ correspond to the $2\times 2$ minors of the matrix
\[ \begin{pmatrix}
x_1 & x_2 & \cdots & x_e\\
x_0 & x_1 & \cdots & x_{e-1}
\end{pmatrix}. \]

We can get the relations between the $Q_{ij}$ as the $3\times 3$ minors of the matrices
\[ \begin{pmatrix}
x_1 & x_2 & \cdots & x_e\\
x_1 & x_2 & \cdots & x_e\\
x_0 & x_1 & \cdots & x_{e-1}
\end{pmatrix} \text{ and } \begin{pmatrix}
x_0 & x_1 & \cdots & x_{e-1}\\
x_1 & x_2 & \cdots & x_e\\
x_0 & x_1 & \cdots & x_{e-1}
\end{pmatrix}. \]

Let $R_n$ be the rational normal curve of degree $n$ in $\P^n$. From the relations above, Proposition 2.4 in \cite{CR} shows that the quadrics $Q_{i,i+1}$, for $1\le i\le n-1$ suffice to determine the elements of $H^0(N_{R_n/\P^n})$.

\begin{proposition}\label{relations}(\cite{CR} Proposition 2.4)
An element $\alpha \in H^0(N_{R_n/\P^n}) = \Hom (\cI_{R_n/\P^n}, \cO_C)$ is determined by the images $\alpha (Q_{i,i+1})$, for $1\le i\le n-1$. Furthermore, $s^{n-i-1}t^{i-1}$ divides $\alpha (Q_{i,i+1})$ and this is the only constraint on $\alpha (Q_{i,i+1})$. If $b_{i,i+1}$, for $1\le i\le n-1$, are arbitrary polynomials of degree $n+2$, there exists an element $\alpha \in H^0(N_{R_n/\P^n})$ such that $\alpha (Q_{i,i+1}) = s^{n-i-1}t^{i-1}b_{i,i+1}$.

In addition, the image $\alpha (Q_{i,j})$ of the other generators of $I_{R_n}$ are expressed in terms of $b_{l,l+1}$ by
\[ \alpha (Q_{i,j}) = \sum _{l=i}^{j-1} s^{n-j-i+l}t^{j+i-l-2}b_{l,l+1}. \]
\end{proposition}

\begin{corollary}\label{normalbundlee} (\cite{CR} Corollary 2.6)
For an integer $e\le n$, the normal bundle $N_{R_e/\P^n}$ is $N_{R_e/\P^e}\oplus N_{\P^e/\P^n}\cong \cO_{\P^1}(e+2)^{e-1}\oplus \cO_{\P^1}(e)^{n-e}$.
\end{corollary}

\subsection{Normal bundles on hypersurfaces} Let $C$ be a smooth rational curve of degree $e$ and let $X$ be a degree $d$ hypersurface in $\P^n$ containing $C$. Using the identification $N_{X/\P^n}\cong \cO_{X}(d)$, we can write the standard normal bundle sequence as
\[ 0\longrightarrow N_{C/X}\longrightarrow N_{C/\P^n}\overset{\psi}{\longrightarrow} N_{X/\P^n}|_C\cong \cO_{\P^1}(de). \]

Thus, we get a map
\[ \phi : H^0(\cI_C(d))\to \Hom (N_{C/\P^n}, \cO_{\P^1}(de)) \]
that sends polynomials $F$ defining $X = V(F)$ to elements $\psi$ in $\Hom (N_{C/\P^n}, \cO_{\P^1}(de))$.

This can be identified via the sequence
\[ 0\longrightarrow \cI_{C/\P^n}^2\longrightarrow \cI_{C/\P^n}\longrightarrow N_{C/\P^n}^*\longrightarrow 0 \]
when we twist it by $\cO_{\P^n}(d)$ and take global sections:
\begin{equation*}\label{phisequence}
0\longrightarrow H^0(\cI_{C/\P^n}^2(d))\longrightarrow H^0(\cI_{C/\P^n}(d))\overset{\phi}{\longrightarrow} H^0(N_{C/\P^n}^*(d)).
\end{equation*}

When $C = R_e$ is the rational normal curve of degree $e$ in $\P^n$ and $d\ge 3$, I. Coskun and E. Riedl showed in \cite{CR} that the map $\phi$ is surjective, that is, every element $\psi \in \Hom (N_{C/\P^n}, \cO_{\P^1}(de))$ can be obtained as the map of a normal bundle sequence for some hypersurface $X$. It is not surjective when $d=2$. In section 4, we compute the dimension of the image of $\phi$ and show that it is injective for $d=2$ and $e=n$.

The relations given in Proposition \ref{relations} allow us to explicitly write the map $\psi _F: N_{C/\P^n}\to \cO_{\P^1}(de)$ for a hypersurface $X=V(F)$. First, let $e=n$. Write $F$ in terms of the generators of $I_C$, $F = \sum _{1\le i<j\le n}F_{i,j}Q_{i,j}$. Then $\psi_F(\alpha) = \sum_{1\le i<j\le n}F_{i,j}|_C\cdot \alpha(Q_{i,j})$. By the relation from Proposition \ref{relations}, we get
\[ \psi_F(\alpha) = \sum _{1\le i<j\le n}F_{i,j}|_C\sum _{l=i}^{j-1} s^{n-j-i+l}t^{j+i-l-2}b_{l,l+1}. \tag{1} \]

Collect the terms and write the sum as $\sum_{i=1}^{n-1}C_ib_{i,i+1}$, then the map $\psi_F: \cO_{\P^1}(n+2)^{n-1}\to \cO_{\P^1}(dn)$ is given by the matrix $(C_1, \ \cdots \ , C_{n-1})$.

When $e<n$, by Corollary \ref{normalbundlee} the normal bundle $N_{C/\P^n}$ splits as the direct sum $N_{C/\P^e}\oplus N_{\P^e/\P^n}$. We write $F = \sum _{1\le i<j\le e}F_{i,j}Q_{i,j} + \sum _{k=e+1}^nG_kx_k$, and collect the coefficients $C_1, \ldots , C_{e-1}$ of the $b_{l,l+1}$ as above. Then the map $\psi_F: \cO_{\P^1}(e+2)^{e-1}\oplus \cO_{\P^1}(e)^{n-e}\to \cO_{\P^1}(de)$ is given by the matrix $(C_1, \ \cdots \ , C_{e-1}; \ G_{e+1}|_C, \ \cdots \ , G_n|C)$.

\medskip

We can use this description to obtain the map $\psi_F$ from explicit hypersurfaces $X=V(F)$. By choosing the appropriate $\psi_F$, we will then obtain each possible splitting type for $N_{C/X}$.

First, observe that when $C$ is the rational normal curve of degree $e$ in $\P^e$, the restrictions $F|_C$ of hypersurfaces of degree $k\ge 1$ cut out the complete linear series $|\cO_C(k)|$, that is, rational normal curves are projectively normal.

\begin{lemma}\label{projectivelynormal}\cite{ACGH}
For every $k\ge 1$, the map $H^0(\cO_{\P^n}(k))\to H^0(\cO_{C}(k))\cong H^0(\cO_{\P^1}(ek))$, $F\mapsto F|_C$, is surjective.
\end{lemma}

To obtain an $F$ for each polynomial in $H^0(\cO_{\P^1}(ek))$ it suffices to write each monomial as a product of $k$ monomials of degree $e$. For instance, for $F|_C = s^{ek-2}t^2$ we can write $s^{ek-2}t^2 = s^{e(k-1)}(s^{e-2}t^2)$ and choose $F = x_0^{k-1}x_2$.

Before we advance to the general case, let us compute an example of a hypersurface $X$ for which $N_{C/X}$ is balanced in order to display the computation process with a simpler notation. In particular, we recover Corollary 3.8 of \cite{CR} for the rational normal curve of degree $n$ in $\P^n$. We make use of the same method in the proof of Theorem \ref{existence}.

\begin{proposition}\label{balancedexample} (cf. \cite{CR}, Corollary 3.8)
Let $d\ge 2$, and let $X$ be a general hypersurface of degree $d$ containing the rational normal curve $C$ of degree $n$ in $\P^n$. Then the normal bundle $N_{C/X}$ is balanced.
\end{proposition}
\begin{proof}
Since being balanced is an open condition in a family of vector bundles on $\P^1$ with fixed rank and degree, it suffices to show an example of a hypersurface $X$ for each $d$. The idea is to compute the kernel of $\psi_F: \cO(n+2)^{n-1}\to \cO(dn)$ directly from the linear relations between the entries of $\psi_F$, which we call the \textit{column relations} of $\psi_F$.

If $d=2$, consider $F = \sum_{i=1}^{n-1} Q_{i,i+1}$, then
\[ \psi_F = ( s^{n-2},\ s^{n-3}t, \cdots ,\ st^{n-3},\ t^{n-2} ). \]

The columns $C_1, \hdots , C_{n-1}$ of $\psi_F$, satisfy the relations
\[ tC_i - sC_{i+1} = t(s^{n-1-i}t^{i-1}) - s(s^{n-2-i}t^{i}) = 0, \ \ 1\le i\le n-2. \]

These relations define the vectors $K_i = (\alpha_1, \ldots , \alpha_{n-1})$ with $\alpha_i = t$, $\alpha_{i+1} = -s$ and $\alpha_j = 0$ for $j\neq i, i+1$. Let $K$ be the matrix whose columns are $K_i$:
\[ K = \begin{pmatrix}
t & 0 & 0 & \cdots & 0 & 0\\
-s & t & 0 & \cdots & 0 & 0\\
0 & -s & t & \cdots & 0 & 0\\
\vdots &  \vdots & \vdots & \ddots & \vdots & \vdots\\
0 &  0  & 0 & \cdots & t & 0\\
0 &  0  & 0 & \cdots & -s & t\\
0 &  0  & 0 & \cdots & 0 & -s\\
\end{pmatrix}. \]

Then $K$ defines a map
\[ K: \cO(n+1)^{n-2}\to \cO(n+2)^{n-1} \]
whose image, by the column relations above, is in the kernel $N_{C/X}$ of $\psi_F$. Notice also that $K$ has maximum rank, so the map is injective, and thus it factors through the $N_{C/X}$. As $N_{C/X}$ has rank $n-2$, we have $N_{C/X}\cong \cO(n+1)^{n-2}$.

Now, let $d\ge 3$. First, observe that if $N_{C/X}$ is balanced, then it must be  $\cO(A)^{m}\oplus \cO(A+1)^{n-2-m}$ with
\[ A = \left \lfloor \dfrac{n(n-d+1)-2}{n-2} \right \rfloor \text{ and } m = A(n-2)-n(n-d). \]

The approach is to construct $\psi_F$ such that its column relations define an injective map
\[ K: \cO(A)^{m}\oplus \cO(A+1)^{n-2-m}\to \cO(n+2)^{n-1} \]
which, as in the case $d=2$, will be the kernel of $\psi_F$. This means we want to obtain $m$ relations of degree $(n+2)-A$ and $n-2-m$ relations of degree $(n+2)-(A+1)$ between the columns of $\psi_F$. To simplify the analysis, we look at polynomials $F = \sum_{i=1}^{n-1}F_iQ_{i,i+1}$ with $F_i$ monomials. Set $F_1 = x_0^{d-2}$ and $F_{n-1} = x_n^{d-2}$, so $F$ induces a map of the form
\[ \psi_F = ( (s^{n-2})(s^{n(d-2)}), \ (s^{n-3}t)F_2|_C, \cdots ,\ (st^{n-3})F_{n-2}|_C, \ (t^{n-2})(t^{n(d-2)}) ). \]

Since an arbitrary degree $d-2$ monomial $F_i$ has the form $F_i|_C = s^{n(d-2)-\beta_i}t^{\beta_i}$ when restricted to $C$, the map $\psi_F$ has the form
\[ \psi_F = ( (s^{n-2})(s^{n(d-2)}), \ (s^{n-3}t)(s^{n(d-2)-\beta_2}t^{\beta_2}), \cdots ,\ (st^{n-3})(s^{n(d-2)-\beta_{n-2}}t^{\beta_{n-2}}), \ (t^{n-2})(t^{n(d-2)}) ). \]

To further simplify, we look for relations that only involve two consecutive entries of $\psi_F$. One way to find such $\psi_F$ is to choose the $\beta_i$ in non-decreasing order $0\le \beta_2\le \cdots \le \beta_{n-2}\le n(d-2)$. This way, we get the columns relations
\begin{align*}
    & t^{\beta_2+1}C_1 - s^{\beta_2+1}C_2 = 0, \\
    & t^{\beta_{i+1}-\beta_i+1}C_i - s^{\beta_{i+1}-\beta_i+1}C_{i+1} = 0, \ \ 2\le i\le n-3 \\
    & t^{n(d-2)-\beta_{n-2}+1}C_{n-2} - s^{n(d-2)-\beta_{n-2}+1}C_{n-1} = 0.
\end{align*}

And as before, we define the matrix $K$ whose columns follow from the relations above:
\[ K = \begin{pmatrix}
t^{\beta_2+1} & 0 & 0 & \cdots & 0 & 0\\
-s^{\beta_2+1} & t^{\beta_3-\beta_2+1} & 0 & \cdots & 0 & 0\\
0 & -s^{\beta_3-\beta_2+1} & t^{\beta_4-\beta_3+1} & \cdots & 0 & 0\\
\vdots &  \vdots & \vdots & \ddots & \vdots & \vdots\\
0 &  0  & 0 & \cdots & t^{\beta_{n-2}-\beta_{n-3}+1} & 0\\
0 &  0  & 0 & \cdots & -s^{\beta_{n-2}-\beta_{n-3}+1} & t^{n(d-2)-\beta_{n-2}+1}\\
0 &  0  & 0 & \cdots & 0 & -s^{n(d-2)-\beta_{n-2}+1}\
\end{pmatrix}\]

Since $K$ has maximum rank, it defines an injective map 
\[ \cO(n+2-(\beta_2+1))\oplus \left (\bigoplus_{i=2}^{n-3}\cO(n+2-(\beta_{i+1}+\beta_i+1))\right )\oplus \cO(n+2-(n(d-2)-\beta_{n-2}+1))\overset{K}{\to}\cO(n+2)^{n-1} \]
which, due to the relations, factors through the kernel of $\psi_F$. Since the direct sum above and $N_{C/X}$ have the same rank, it follows that
\[ N_{C/X}\cong \cO(n+2-(\beta_2+1))\oplus \left (\bigoplus_{i=2}^{n-3}\cO(n+2-(\beta_{i+1}+\beta_i+1))\right )\oplus \cO(n+2-(n(d-2)-\beta_{n-2}+1)). \tag{$2$} \]

To conclude, we pick the appropriate values for $\beta_i$, $2\le i\le n-2$. Choose
\[ \beta _i = \left\{\begin{matrix}
(i-1)(n+1-A), & 2\le i\le m+1\\ 
(i-1)(n-A)+m, & m+2\le i\le n-2
\end{matrix}\right., \]
noting that, by Lemma \ref{projectivelynormal}, there exist $F_i$ such that $F_i|_C = s^{n(d-2)-\beta_i}t^{\beta_i}$. A simple computation then shows that the direct sum on the right-hand side of $(2)$ is $\cO(A)^{m}\oplus \cO(A+1)^{n-2-m}$.
\end{proof}

\section{Hypersurfaces of Degree at Least 3}

Throughout this section, we assume $d\ge 3$. We study the splitting types of $N_{C/X}$ for hypersurfaces of degree $d$ in $\P^n$ containing a degree $e\le n$ rational normal curve $C = R_e$. By Corollary \ref{normalbundlee}, the normal bundle $N_{C/X}$ is the kernel in the sequence
\[ 0\longrightarrow N_{C/X}\longrightarrow \cO(e+2)^{e-1}\oplus \cO(e)^{n-e}\overset{\psi_F}{\longrightarrow} \cO(de)\longrightarrow 0 \]
for $F$ a degree $d$ polynomial defining $X = V(F)$.

In this section, we ask what splitting types occur as the normal bundle $N_{C/X}$, and look for examples of hypersurfaces $X$ for each possible splitting type. We also ask for the dimension of the space $\Sigma_{\vec{a}}$ of hypersurfaces $X$ with a given splitting type $E_{\vec{a}}$.

The surjectivity of the map $\phi : H^0(\cI_{C/\P^n}(d))\to \Hom (N_{C/\P^n}, \cO(de))$ for $d\ge 3$ (\cite{CR} Theorem 3.1) is enough to show that $N_{C/X}$ is balanced for the general $X$ (\cite{CR}, Corollary 3.8), and we use it here to compute the dimension of $\Sigma_{\vec{a}}$.

\subsection{Examples of hypersurfaces for each splitting type}

From the short exact sequence, a candidate for splitting type of $N_{C/X}$ must have the form of a rank $(n-2)$ direct sum 
\[ N_{C/X}\cong \left(\bigoplus_{i=1}^{e-1}\cO(e+2-a_i)\right )\oplus \left (\bigoplus_{j=e+2}^{n}\cO(e-b_j) \right ), \]
with $a_i, b_j\ge 0$, and degree $\deg N_{C/X} = \deg N_{C/\P^n} - \deg \cO(de) = e(n-d+1)-2$, equivalently $\sum_{i=1}^{e-1}a_i + \sum_{j=e+2}^n b_j = e(d-1)$.

\medskip
In the following theorem, we find examples of hypersurfaces $X$, smooth along $C$, for each splitting type with at most $e-2$ summands of degrees greater than $e$. If the normal bundle $N_{C/X}$ has $e-1$ summands of degree greater than $e$, then the normal bundle has a summand $\cO(e+2)^{e-1}$. The proof of the theorem is a generalization of the computation from Proposition \ref{balancedexample}.

\begin{theorem}\label{existence}
\leavevmode
\begin{enumerate}[label=(\alph*)]
    \item For all splitting types with at most $e-2$ summands of degrees greater than $e$, that is, of the form
    \[ E = \left(\bigoplus_{i=1}^{e-2}\cO(e+2-a_i)\right )\oplus \left (\bigoplus_{j=e+1}^{n}\cO(e-b_j) \right ), \]
    with $a_i, b_j\ge 0$ and $\sum_{i=1}^{e-2}a_i + \sum_{j=e+1}^n b_j = e(d-1)-2$, we obtain explicit  examples of degree $d$ hypersurfaces $X$, smooth along the curve $C$, with normal bundle $N_{C/X}\cong E$.  
\newline
    \item If $N_{C/X}$ is not of the form $E$ above, then $e<n-1$, and the normal bundle must have the form
    \[ E' = \cO(e+2)^{e-1}\oplus \left (\bigoplus_{j=e+2}^{n}\cO(e-b_j) \right ), \]
    with $b_j\ge 0$ and $\sum_{j=e+2}^n b_j = e(d-1)$.
    For each splitting type $E'$, we find explicit examples of degree $d$ hypersurfaces $X$, smooth along $C$, with normal bundle $N_{C/X}\cong E'$.
\end{enumerate}

\end{theorem}

\begin{proof}
(a) We divide the proof into the cases $e=n$ and $e<n$.

For $e=n$, we ask if we can get all splitting types $E = \bigoplus_{i=1}^{n-2}\cO(n+2-a_i)$ with $a_i\ge 0$ and $\sum _{i=1}^{n-2}a_i = n(d-1)-2$. The approach is to construct a map $\psi_F$ from which we can easily compute the kernel from the linear relations between its entries, which we call the \textit{column relations} of $\psi_F$.

In order to start with a simple $\psi_F$, remember that $Q_{i,i+1} = x_i^2 - x_{i-1}x_{i+1}$, and consider polynomials $F$ of the form $F = \sum_{i=1}^{n-1}F_iQ_{i,i+1}$ with $F_1 = x_0^{d-2}$ and $F_{n-1} = x_n^{d-2}$, so the map $\psi_F: \cO(n+2)^{n-1}\to \cO(dn)$ has the form
\[ \psi_F = (s^{n-2}(s^{n(d-2)}), \ \ (s^{n-3}t)F_2|_C, \ \cdots ,\ (st^{n-3})F_{n-2}|_C, \ \ t^{n-2}(t^{n(d-2)})). \]

First, we check that these are all smooth along the curve $C$. Taking partial derivatives $\frac{\partial F}{\partial x_i}$ and restricting to the curve (that is, taking $Q_{ij} = 0$), we obtain
\[ \frac{\partial F}{\partial x_0} = -(d-1)x_0^{d-2}x_2, \ \ \frac{\partial F}{\partial x_{2}} = -x_0^{d-1} + 2x_{2}F_{2} - x_{4}F_{3} \ \ \text{ and } \ \ \frac{\partial F}{\partial x_{n-2}} = -x_n^{d-1} + 2x_{n-2}F_{n-2} - x_{n-4}F_{n-3}. \]
If $F$ is singular at a point $P = (s^n : s^{n-1}t : \cdots : t^n)$ on $C$, then $\frac{\partial F}{\partial x_0} = 0$ gives $s=0$ or $t=0$. If $s=0$, then $\frac{\partial F}{\partial x_{n-2}}=0$ imples $t=0$; and if $t=0$, then $\frac{\partial F}{\partial x_2}=0$ implies $s=0$. Thus, polynomials $F$ of this form give hypersurfaces that are smooth along $C$.

Now, look at the map $\psi_F$. One way to obtain a $\psi_F$ with clear relations between its entries is to choose increasing exponents for $t$. This way, we get a relation between every two consecutive entries by multiplying the first by a power of $t$ and the second by a power of $s$. We will also need zero entries to get the terms $\cO(n+2)$ of $E$. So, let $0\le z\le n-3$ and  $0\le \beta_2\le \beta _3\le \cdots \le \beta_{n-z-2}\le n(d-2)+z+1$ be integers, and let $F_i$ be such that
\[ F_i|_C = s^{n(d-2)-\beta_i}t^{\beta_i} \text{ for } 2\le i\le n-z-2, \]
and
\[ F_i|_C = 0 \text{ for } n-z-1\le i\le n-2. \]

Note that such $F_i$ exist and can be easily obtained by Lemma \ref{projectivelynormal}. Then, $\psi_F$ has increasing powers of $t$ and $z$ zero entries:
\begin{multline*}
 \psi_F = ( s^{n-2}(s^{n(d-2)}), \ (s^{n-3}t)(s^{n(d-2)-\beta_2}t^{\beta_2}), \cdots , (s^{z+1}t^{n-z-3})(s^{n(d-2)-\beta_{n-z-2}}t^{\beta_{n-z-2}}), \ 0, \ 0, \cdots \\
 \cdots 0, \ t^{n-2}(t^{n(d-2)}) ).
\end{multline*}

Its columns $C_1, C_2, \ldots , C_{n-1}$ satisfy the column relations
\begin{align*}
   & t^{\beta_2+1}C_1 - s^{\beta_2+1}C_2 = 0, \\
    & t^{\beta_{i+1}-\beta_i+1}C_i - s^{\beta_{i+1}-\beta_i+1}C_{i+1} = 0, \ \ 2\le i\le n-z-3, \\
    & C_j = 0, \ \ n-z-1\le j\le n-2, \\
    & t^{n(d-2)+z+1-\beta_{n-z-2}}C_{n-z-2} - s^{n(d-2)+z+1-\beta_{n-z-2}}C_{n-1} = 0.
\end{align*}

To simplify notation and make the final result simpler to parse, we write $\gamma_i = \beta_{i+1} - \beta_i$ for $2\le i\le n-z-3$, $\gamma_1 = \beta_2$ and $\gamma_{n-z-2} = \beta_{n-z-2}$. The conditions for $\beta_i$ translate into the conditions $\gamma_i\ge 0$ and $\gamma_{n-z-2} = \sum_{i=1}^{n-z-3}\gamma_i\le n(d-2)+z+1$. Rewriting the column relations:
\begin{align*}
   & t^{\gamma_1+1}C_1 - s^{\gamma_1+1}C_2 = 0, \\
    & t^{\gamma_i+1}C_i - s^{\gamma_i+1}C_{i+1} = 0, \ \ 2\le i\le n-z-3, \\
    & C_j = 0, \ \ n-z-1\le j\le n-2, \\
    & t^{n(d-2)+z+1-\gamma_{n-z-2}}C_{n-z-2} - s^{n(d-2)+z+1-\gamma_{n-z-2}}C_{n-1} = 0.
\end{align*}

We define the matrix $K$ whose columns give the relations above:
\[ K = \begin{pmatrix}
t^{\gamma_1+1} & 0 & 0 & \cdots & 0 & 0 & 0 & \cdots & 0 & 0 \\
-s^{\gamma_1+1} & t^{\gamma_2+1} & 0 & \cdots & 0 & 0 & 0 & \cdots & 0 & 0 \\
0 & -s^{\gamma_2+1} & t^{\gamma_3+1} & \cdots & 0 & 0 & 0 & \cdots & 0 & 0 \\
\vdots &  \vdots & \vdots & \ddots & \vdots & \vdots & \vdots & \ddots & \vdots & \vdots\\
0 &  0  & 0 & \cdots & t^{\gamma_{n-z-3}+1} & 0 & 0 & \cdots & 0 & 0\\
0 &  0  & 0 & \cdots & -s^{\gamma_{n-z-3}+1} & 0 & 0 & \cdots & 0 & t^{n(d-2)+z+1-\gamma_{n-z-2}}\\
0 &  0  & 0 & \cdots & 0 & 1 & 0 &\cdots & 0 & 0\\
0 &  0  & 0 & \cdots & 0 & 0 & 1 & \cdots & 0 & 0\\
\vdots &  \vdots & \vdots & \ddots & \vdots & \vdots & \vdots & \ddots & \vdots & \vdots \\
0 &  0  & 0 & \cdots & 0 & 0 & 0 & \cdots & 1 & 0\\
0 &  0  & 0 & \cdots & 0 & 0 & 0 & \cdots & 0 & -s^{n(d-2)+z+1-\gamma_{n-z-2}}
\end{pmatrix}. \]

Then $K$ defines a map
\[ K: \left (\bigoplus_{i=1}^{n-z-3} \cO(n+1-\gamma_i)\right )\oplus \cO(n+2)^z\oplus \cO(n+1 - n(d-2)-z+\gamma_{n-z-2})\to \cO(n+2)^{n-1}. \]

The column relations imply that the image of $K$ is contained in the kernel $N_{C/X}$ of $\psi_F$. Since $K$ has maximal rank, its image is a subbundle of $N_{C/X}$. As both $N_{C/X}$ and the subbundle have rank $n-2$, it follows that
\[ N_{C/X}\cong \cO(n+2)^z\oplus \left (\bigoplus_{i=1}^{n-z-3} \cO(n+1-\gamma_i)\right )\oplus \cO(n+1 - n(d-2)-z+\gamma_{n-z-2}).\]

Thus, we get splitting types of $N_{C/X}$ with exactly $z$ terms $\cO(n+2)$. By appropriately choosing the $\gamma_i\ge 0$ with $\gamma_{n-z-2} = \sum_{i=1}^{n-z-3}\gamma_i\le n(d-2)+z+1$, or equivalently, by choosing the $\beta_i$ in non-decreasing order, we get all possible choices for the summands of degree smaller than $n+2$. Note the last term is determined by the previous ones and the degree of $N_{C/X}$. Finally, by varying the number of zero entries $z$ between $0$ and $n-3$, we obtain all splitting types $E$ for $N_{C/X}$.

\medskip

Now, we investigate the case $e<n$. Similarly, consider polynomials $F$ of the form $F = \sum _{i=1}^{e-1}F_iQ_{i,i+1} + \sum _{j=e+1}^nG_jx_j$ with $F_1 = x_0^{d-2}$ and $G_n = x_e^{d-1}$, so that $N_{C/X}$ is the kernel of maps $\psi_F:\cO(e+2)^{e-1}\oplus \cO(e)^{n-e}\to \cO(de)$ of the form
\[ \psi_F = \left ( s^{e-2}(s^{e(d-2)}), \ (s^{e-3}t)F_2|_C, \cdots , t^{e-2}F_{e-1}|_C \ ; \ G_{e+1}|_C, \ \cdots , G_{n-1}|_C, \ t^{e(d-1)} \right ). \]

The approach in this case is similar. The main difference is that $N_{C/\P^n}$ splits into two summands of different degrees: $\cO(e+2)^{e-1}\oplus \cO(e)^{n-e}$. We separate $\psi_F$ into two parts as above, corresponding to the $F_i$ and $G_j$. We will choose the zero entries all in the first part so we can get the $\cO(e+2)$ terms of $E$.

First, we show that these $F$ define hypersurfaces smooth along the curve. Taking partial derivatives $\frac{\partial F}{\partial x_i}$ and restricting to the curve (letting $Q_{ij} = 0$ and $x_l = 0$, $l\ge e+1$), we obtain
\[ \frac{\partial F}{\partial x_2} = -x_0^{d-1} + 2x_2F_2 - x_4F_3 \ \ \ \text{ and } \ \ \ \frac{\partial F}{\partial x_n} = x_e^{d-1}. \]

So, if $F$ is singular at a point $P=(s^e : s^{e-1}t : \cdots : t^e : 0 : \cdots : 0)$ on $C$ then $x_e^{d-1} = 0$ implies $t=0$, thus $x_2 = x_4 =0$, and the first derivative gives $x_0^{d-1} = 0$, so $s=0$. Hence, all $F$ of this form are smooth along $C$.

As in the case $e=n$, we choose maps $\psi_F$ with $z$ zero entries and increasing exponents of $t$ to get relations between every two consecutive entries. This time, consider integers
\[ 0\le z\le e-2, \ \ 0\le \beta_2\le \beta_3\le \cdots \le \beta_{e-1-z} \ \text{ and } \ \beta_{e-1-z}-z+e\le \beta_e\le \cdots \le \beta_{n-2}\le e(d-1), \]
and let $F_i$ and $G_j$ be such that
\begin{align*}
& F_i|_C = s^{e(d-2)-\beta_i}t^{\beta_i} \text{ for } 2\le i\le e-1-z,\\
& F_i|_C = 0 \text{ for } e-z\le i\le e-1 \text{ and }\\
& G_j|_C = s^{e(d-1)-\beta_{j-1}}t^{\beta_{j-1}} \text{ for } e+1\le j\le n-1.
\end{align*}
so the map $\psi_F$ has the form
\begin{multline*}
\psi_F =  ( s^{e-2}(s^{e(d-2)}), \ \ (s^{e-3}t)(s^{e(d-2)-\beta_2}t^{\beta_2}), \ \cdots , \ (s^zt^{e-2-z})(s^{e(d-2)-\beta_{e-1-z}}t^{\beta_{e-1-z}}), \ \ 0, \ \cdots , \ 0 ; \\
(s^{e(d-1)-\beta_e}t^{\beta_e}), \ \ (s^{e(d-1)-\beta_{e+1}}t^{\beta_{e+1}}), \ \cdots , \ (s^{e(d-1)-\beta_{n-2}}t^{\beta_{n-2}}), \ \ t^{e(d-1)} ),
\end{multline*}

We obtain the following relations between the columns $C_1, \ldots , C_{e-1}; C_e, \ldots C_{n-1}$:
\begin{align*}
    & t^{\beta_2+1}C_1 - s^{\beta_2+1}C_2 = 0\\
    & t^{\beta_{i+1}-\beta_i+1}C_i - s^{\beta_{i+1}-\beta_i+1}C_{i+1} = 0, \text{ for } 2\le i\le e-2-z\\
    & C_i = 0, \text{ for } e-z\le i\le e-1\\
    & t^{\beta_e-\beta_{e-1-z}-e+2+z}C_{e-1-z} - s^{\beta_e-\beta_{e-1-z}+z-e}C_{e} = 0\\
    & t^{\beta_{j+1}-\beta_j}C_j - s^{\beta_{j+1}-\beta_j}C_{j+1} = 0, \text{ for } e\le j\le n-3\\
    & t^{e(d+1)-\beta_{n-2}}C_{n-2}-s^{e(d+1)-\beta_{n-2}}C_{n-1} = 0.
\end{align*}

As before, rename $\gamma_i = \beta_{i+1} - \beta_i$, $\gamma_1 = \beta_2$, $\gamma_{e-1} = \beta_e - \beta_{e-1-z} - e + z$ and $\gamma_{n-2} = \beta_{n-2}$. These satisfy the conditions $\gamma_i\ge 0$ for all $i$ and $\gamma_{n-2} = \sum_{i=1}^{n-2-z}\gamma_i + \sum_{j=e-1}^{n-3}\gamma_j\le e(d-1)$. Then the column relations become
\begin{align*}
    & t^{\gamma_2+1}C_1 - s^{\gamma_2+1}C_2 = 0\\
    & t^{\gamma_i+1}C_i - s^{\gamma_i+1}C_{i+1} = 0, \text{ for } 2\le i\le e-2-z\\
    & C_i = 0, \text{ for } e-z\le i\le e-1\\
    & t^{\gamma_{e-1}+2}C_{e-1-z} - s^{\gamma_{e-1}}C_{e} = 0\\
    & t^{\gamma_j}C_j - s^{\gamma_j}C_{j+1} = 0, \text{ for } e\le j\le n-3\\
    & t^{e(d+1)-\gamma_{n-2}}C_{n-2}-s^{e(d+1)-\gamma_{n-2}}C_{n-1} = 0.
\end{align*}

They induce the matrix

\[ K = \begin{pmatrix}
t^{\gamma_1+1} & 0 & \cdots & 0 & 0 & 0 & \cdots & 0 & 0 & 0 & \cdots & 0 & 0 \\
-s^{\gamma_1+1} & t^{\gamma_2+1} & \cdots & 0 & 0 & 0 & \cdots & 0 & 0 & 0 & \cdots & 0 & 0 \\
0 & -s^{\gamma_2+1} & \cdots & 0 & 0 & 0 & \cdots & 0 & 0 & 0  & \cdots & 0 & 0 \\
\vdots &  \vdots & \ddots & \vdots & \vdots & \vdots & \ddots & \vdots & \vdots & \vdots & \ddots & \vdots & \vdots \\
0 &  0  & \cdots & t^{\gamma_{e-z-2}+1} & 0 & 0 & \cdots & 0 & 0 & 0 & \cdots & 0 & 0\\
0 &  0  & \cdots & -s^{\gamma_{e-z-2}+1} & 0 & 0 & \cdots & 0 & t^{\gamma_{e-1}+2} & 0 & \cdots & 0 & 0\\
0 &  0  & \cdots & 0 & 1 & 0 &\cdots & 0 & 0 & 0 & \cdots & 0 & 0\\
0 &  0  & \cdots & 0 & 0 & 1 & \cdots & 0 & 0 & 0 & \cdots & 0 & 0\\
\vdots &  \vdots & \ddots & \vdots & \vdots & \vdots & \ddots & \vdots & \vdots  & \vdots & \ddots & \vdots & \vdots\\
0 &  0  & \cdots & 0 & 0 & 0 & \cdots & 1 & 0 & 0 & \cdots & 0 & 0\\
0 &  0  & \cdots & 0 & 0 & 0 & \cdots & 0 & -s^{\gamma_{e-1}} & t^{\gamma_e} & \cdots & 0 & 0 \\
0 &  0  & \cdots & 0 & 0 & 0 & \cdots & 0 & 0 & -s^{\gamma_e} & \cdots & 0 & 0 \\
\vdots &  \vdots & \ddots & \vdots & \vdots & \vdots & \ddots & \vdots & \vdots  & \vdots & \ddots & \vdots & \vdots\\
0 &  0  & \cdots & 0 & 0 & 0 & \cdots & 0 & 0 & 0 & \cdots & t^{\gamma_{n-3}} & 0\\
0 &  0  & \cdots & 0 & 0 & 0 & \cdots & 0 & 0 & 0 & \cdots & -s^{\gamma_{n-3}} & t^{e(d+1)-\gamma_{n-2}}\\
0 &  0  & \cdots & 0 & 0 & 0 & \cdots & 0 & 0 & 0 & \cdots & 0 & -s^{e(d+1)-\gamma_{n-2}}
\end{pmatrix}. \]

\medskip

And $K$ defines an injective map
\[ K: \left( \bigoplus _{i=1}^{e-2-z}\cO(e+1-\gamma_i) \right )
\oplus \cO(e+2)^{z}\oplus \left ( \bigoplus _{j=e-1}^{n-3}\cO(e-\gamma_j) \right ) \oplus \cO(e-e(d+1)+\gamma_{n-2})\to \cO(e+2)^{e-1}\oplus \cO(e)^{n-e}. \]

As before, the column relations imply that $K$ factors through the kernel of $\psi_F$, and consequently, the kernel is
\[ N_{C/X}\cong \cO(e+2)^{z}\oplus \left( \bigoplus _{i=1}^{e-2-z}\cO(e+1-\gamma_i) \right )
\oplus \left ( \bigoplus _{j=e-1}^{n-3}\cO(e-\gamma_j) \right ) \oplus \cO(e-e(d+1)+\gamma_{n-2}). \]

This is a splitting type with exactly $z$ terms $\cO(e+2)$, $e-2-z$ terms of degree less or equal than $e+1$, and the rest of them of degree less or equal than $e$. By varying the $\gamma_i$ with $\gamma_i\ge 0$ and $\gamma_{n-2} = \sum_{i=1}^{n-2-z}\gamma_i + \sum_{j=e-1}^{n-3}\gamma_j\le e(d-1)$, we get all splitting types $E$ of this form. Finally, by varying the number of zero entries $z$ between $0$ and $e-2$ we obtain all possible splitting types $E$ for $N_{C/X}$.

\medskip

(b) If $N_{C/X}$ does not have splitting type of the form $E$, then it has more than $e-2$ summands of degree greater than $e$. Remember from Section 2, that the map $\psi_F: \cO_{\P^1}(e+2)^{e-1}\oplus \cO_{\P^1}(e)^{n-e}\to \cO_{\P^1}(de)$ is given by a matrix $(C_1, \ \cdots \ , C_{e-1}; \ G_{e+1}|_C, \ \cdots \ , G_n|C)$. Consider this matrix divided into two parts: the first corresponding to the $C_i$, and the second to the $G_j|_C$.

As we have seen in the proof of part (a), any summand of $N_{C/X}$ of degree greater than $e$ must come from column relations involving only entries of the first part of $\psi_F$. However, since the first part has $e-1$ entries, we can obtain at most $e-2$ linearly independent relations from it, unless it is zero. Thus, we must have $C_i = 0$ for $i=1,\hdots , e-1$, and
\[ \psi_F = \left ( 0, \ 0, \cdots , 0 \ ; \ G_{e+1}|_C, \ \cdots , G_{n-1}|_C, \ G_n|_C \right ). \]
In this case, we get the column relations $C_i = 0$ for $1\le i\le e-1$, which will correspond to the summands $\cO(e+2)^{e-1}$. The remaining $n-e-1$ column relations will come from the second part of $\psi_F$. If $e=n-1$, this would mean $N_{C/X}\cong \cO(e+2)^{e-1}$, which cannot happen for degree reasons. Hence, $e<n-1$.

To simplify our approach, consider polynomials of the form $F = \sum _{j=e+1}^nG_jx_j$. Taking partial derivatives and restricting to $C$, we have
\[ \frac{\partial F}{\partial x_a} = \sum_{j=e+1}^n G_j|_C\frac{\partial x_j}{\partial x_a} = \left\{\begin{matrix}
0 \ \text{ if } a\le e-1\\ 
G_a|_C \ \text{ if } a\ge e+1
\end{matrix}\right., \]
so $F$ is singular at the intersection of the zero locus of the $G_j|_C$, $e+1\le j\le n$. We get explicit examples below.

Let $0\le \beta_{e+2}\le \beta_{e+3}\le \cdots \le \beta_{n-1}\le e(d-1)$, and choose $G_{e+1} = x_0^{d-1}$, $G_n = x_e^{d-1}$ and the other $G_j$ so that
\begin{align*}
    & G_{e+1}|_C = s^{e(d-1)}\\
    & G_j|_C = s^{e(d-1)-\beta_j}t^{\beta_j}, \text{ for } e+2\le j\le n-1\\
    & G_{n} = t^{e(d-1)}.
\end{align*}

First, notice these hypersurfaces $F$ are smooth along $C$, since $G_{e+1}|_C = G_n|_C = 0$ implies $s=t=0$.

For these $F$, the map $\psi_F$ has the form
\[ \psi_F = \left ( 0, \ 0, \cdots , 0 \ ; \ s^{e(d-1)}, \ s^{e(d-1)-\beta_{e+2}}t^{\beta_{e+2}} , \ \cdots , s^{e(d-1)-\beta_{n-1}}t^{\beta_{n-1}}, t^{e(d-1)} \right ). \]

To simplify notation, rename $\gamma_{j+1} = \beta_{j+1}-\beta_j$, $\gamma_{e+2} = \beta_{e+2}$ and $\gamma_{n} = e(d-1) - \beta_{n-1}$. Then, we get the following relations between the columns $C_1, \hdots , C_{e-1}; C_{e+1}, \hdots , C_n$ of $\psi_F$:
\begin{align*}
    & C_i = 0 \ \text{ for } 1\le i\le e-1\\
    & t^{\gamma_{j+1}}C_{j} - s^{\gamma_{j+1}}C_{j+1} = 0 \ \text{ for } e+1\le j\le n-1.
\end{align*}

Writing these relations in column vectors, we define the matrix
\[ K = \begin{pmatrix}
1 & 0 & \cdots & 0 & 0 & 0 & 0 & \cdots & 0 \\
0 & 1 & \cdots & 0 & 0 & 0 & 0 & \cdots & 0 \\
\vdots & \vdots  & \ddots & 0 & 0 & 0 & 0 & \cdots & 0 \\
0 & 0 & \cdots & 1 & 0 & 0 & 0 & \cdots & 0 \\
0 & 0 & \cdots & 0 & t^{\gamma_{e+2}} & 0 & 0 & \cdots & 0 \\
0 & 0 & \cdots & 0 & -s^{\gamma_{e+2}} & t^{\gamma_{e+3}} & 0 & \cdots & 0  \\
0 & 0 & \cdots & 0 & 0 & -s^{\gamma_{e+3}} & t^{\gamma_{e+4}} & \cdots & 0 \\
\vdots & \vdots & \vdots & \vdots & \vdots &  \vdots & \vdots & \ddots & \vdots\\
0 & 0 & 0 & 0 & 0 &  0  & 0 & \cdots & t^{\gamma_{n}} \\
0 & 0 & 0 & 0 & 0 &  0  & 0 & \cdots & -s^{\gamma_{n}}
\end{pmatrix}.\]

It induces the map
\[ K: \cO(e+2)^{e-1}\oplus \left ( \bigoplus_{j=e+2}^{n}\cO(e-\gamma_j) \right )\to \cO(e+2)^{e-1}\oplus \cO(e)^{n-e}. \]

The column relations imply that the image of $K$ is contained in the kernel $N_{C/X}$ of $\psi_F$. Thus, since $K$ has maximal rank $n-2$, it follows that
\[ N_{C/X}\cong \cO(e+2)^{e-1}\oplus \left ( \bigoplus_{j=e+2}^{n}\cO(e-\gamma_j) \right ). \]
By varying our choices of $\beta_i$, we can obtain all splitting types $E'$.
\end{proof}

\begin{corollary}\label{smoothexample}
(char $K=0$) Let $d\ge 4$, and assume the base field $K$ has characteristic $0$. For all splitting types $E$ and $E'$ as in the theorem, there exists a smooth hypersurface $X$ of degree $d$ containing the curve $C$ with normal bundle $N_{C/X}$ with that splitting type.
\end{corollary}
\begin{proof} Let $\psi_F\in \Hom (N_{C/\P^n},\cO(de))$ be the map with $N_{C/X}\cong E$, induced by the polynomial $F$ obtained in the theorem. From the exact sequence
\[ 0\longrightarrow H^0(\cI_C^2(d))\longrightarrow H^0(\cI_C(d))\overset{\phi}{\longrightarrow} \Hom (N_{C/\P^n},\cO(de)), \]
we have that the space of polynomials that induce the same $\psi_F$ is the coset 
\[ F+H^0(\cI_C^2(d)) = \{ F+G \ | \ G\in H^0(\cI_C^2(d)) \}. \]

If $d\ge 4$, then multiples of $Q_{ij}^2$ are in $H^0(\cI_C^2(d))$, so the base locus of $F + H^0(\cI_C^2(d))$ is the curve $C$. Thus, by Bertini's theorem, the general member of $F+H^0(\cI_C^2(d))$ is smooth away from $C$. As $F$ is smooth along $C$, the general member of $F+H^0(\cI_C^2(d))$ is smooth.
\end{proof}

\subsection{Dimension Count}

Theorem \ref{existence} shows that for all possible splitting types $E$ and $E'$ we can obtain a degree $d\ge 3$ hypersurface $X$ for which $N_{C/X}$ has the given splitting type. A natural question that arises is if we can compute the dimension of the space of such hypersurfaces, and if it is the expected one.

Let $E_{\vec{a}} = \bigoplus _{i=1}^{n-2}\cO(a_i)$ be a possible splitting type. If $E_{\vec{a}}$ is of the form $E$, the numerical conditions in Theorem \ref{existence} translate to $a_i\le e+2$ for $1\le i\le e-2$ and $a_j\le e$ for $e-1\le j\le n-2$. If $E'$, we have $a_i=e+2$ for $1\le i\le e-1$ and $a_j\le e$ for $e\le j\le n-2$. In both cases, $\sum_{i=1}^{n-2}a_i = e(n-d+1)-2$.

In the introduction, we defined the spaces
\[ \Sigma = \{ F \ | \ X \text{ is a degree } d \text{ hypersurface smooth along } C \}\subset H^0(\cO_{\P^n}(d)) \]
and
\[ \Sigma_{\vec{a}} = \{ F\in \Sigma \ | \ N_{C/X}\cong E_{\vec{a}} \}\subset \Sigma , \]
and observed that the expected codimension of $\Sigma_{\vec{a}}$ in $\Sigma$ is $h^1(\sEnd(E_{\vec{a}}))$.

We start by computing this codimension at the level of homomorphisms. For that, define
\[ \Phi_{\vec{a}} = \{ M\in \Hom (N_{C/\P^n}, \cO_{\P^1}(de)) \ | \ \ker M\cong E_{\vec{a}} \}\subset \Hom (N_{C/\P^n}, \cO_{\P^1}(de)). \]

\begin{proposition}\label{lemmaHom} The space $\Phi_{\vec{a}}$ is smooth and irreducible of codimension
\[ h^1(\sEnd (E_{\vec{a}})) - h^1(\sHom(E_{\vec{a}}, N_{C/\P^n})) \]
in $\Hom (N_{C/\P^n}, \cO_{\P^1}(de))$.
\end{proposition}
\begin{proof}
We first compute the dimension of $\Phi_{\vec{a}}$. Every $M$ in $\Phi_{\vec{a}}$ comes from an injection $\theta: E_{\vec{a}}\hookrightarrow N_{C/\P^n}$, and every injection defines an $M$, since its cokernel is a line bundle of degree $de$. By our numerical conditions on $\vec{a}$, $\sHom (E_{\vec{a}}, N_{C/\P^n})$ is globally generated, so the dimension of injections $\theta$ is $h^0(\sHom(E_{\vec{a}}, N_{C/\P^n}))$. However, each $M$ might come from different injections. To count these, consider the set of pairs of homomorphisms
\[ B = \{ (\theta, M)\mid E_{\vec{a}}\overset{\theta}{\hookrightarrow} N_{C/\P^n}\overset{M}{\twoheadrightarrow} \cO(de) \}\subset \Hom (E_{\vec{a}}, N_{C/\P^n})\times \Hom (N_{C/\P^n}, \cO(de)) \]
together with the projections $\pi_1$ and $\pi_2$:
\[ \begin{tikzcd}
& B \arrow[ld, "\pi_1"'] \arrow[rd, "\pi_2"] & \\
\Hom (E_{\vec{a}}, N_{C/\P^n}) & & \Hom (N_{C/\P^n}, \cO(de))
\end{tikzcd} \]

A fiber $\pi_1^{-1}(\theta)$ of the map $\pi_1$ corresponds to the $M$'s induced by $\theta$ up to a choice of basis for $\cO(de)$, thus, the fibers of $\pi_1$ are isomorphic to $\Aut(\cO(de))$. The image of $\pi_1$ is the open set of injections $\theta \in \Hom (E_{\vec{a}}, N_{C/\P^n})$. Hence, $B$ is smooth and irreducible of $\dim B = h^0(\sHom(E_{\vec{a}}, N_{C/\P^n})) + h^0(\sEnd(\cO(de)))$.

On the other hand, the image of $\pi_2$ is exactly the set $\Phi_{\vec{a}}$, and the fibers of $\pi_2$ are isomorphic to $\Aut (E_{\vec{a}})$. Thus, $\Phi_{\vec{a}}$ is smooth and irreducible, and
\[ \dim \Phi_{\vec{a}} = h^0(\sHom(E_{\vec{a}}, N_{C/\P^n})) + h^0(\sEnd(\cO(de))) - h^0(\sEnd(E_{\vec{a}})). \]

Therefore, its codimension in $\Hom (N_{C/\P^n}, \cO(de))$ is
\[ \codim \Phi_{\vec{a}} = h^0(\sHom (N_{C/\P^n}, \cO_{\P^1}(de))) - h^0(\sHom(E_{\vec{a}}, N_{C/\P^n})) + h^0(\sEnd(E_{\vec{a}})) - h^0(\sEnd(\cO(de))). \tag{$3$} \]

Now, applying the functors $\sHom (E_{\vec{a}}, - )$ and $\sHom (-, \cO(de))$ to the sequence
\[ 0\longrightarrow E_{\vec{a}}\longrightarrow N_{C/\P^n}\longrightarrow \cO(de)\longrightarrow 0 \]
we obtain the long exact sequences
\begin{align*}
     0\longrightarrow & \sHom(E_{\vec{a}}, E_{\vec{a}})\longrightarrow \sHom(E_{\vec{a}}, N_{C/\P^n})\longrightarrow \sHom(E_{\vec{a}}, \cO(de))\longrightarrow \\
     \longrightarrow & \sExt^1(E_{\vec{a}}, E_{\vec{a}})\longrightarrow \sExt^1(E_{\vec{a}}, N_{C/\P^n})\longrightarrow \sExt^1(E_{\vec{a}}, \cO(de)) = 0
\end{align*}
and
\[ 0\longrightarrow \sHom(\cO(de), \cO(de))\longrightarrow \sHom(N_{C/\P^n}, \cO(de))\longrightarrow \sHom(E_{\vec{a}}, \cO(de))\longrightarrow 0. \]

Observe that $\sExt^1(E_{\vec{a}}, \cO(de)) = 0$ since $a_i\le e+2\le de$ for all $1\le i\le n-1$.
Computing the alternating sum on dimensions of the first sequence, and substituting it on $(3)$, we get
\begin{multline*}
\codim \Phi_{\vec{a}} = h^0(\sHom (N_{C/\P^n}, \cO_{\P^1}(de))) - h^0(\sHom(E_{\vec{a}}, \cO(de))) + h^1(\sEnd(E_{\vec{a}}))\\
- h^1(\sHom(E_{\vec{a}}, N_{C/\P^n})) - h^0(\sEnd(\cO(de))).
\end{multline*}

And from the second sequence it follows that
\begin{align*}
    \codim \Phi_{\vec{a}} = & h^0(\sEnd(\cO(de))) + h^0(\sHom(E_{\vec{a}}, \cO(de))) - h^0(\sHom(E_{\vec{a}}, \cO(de))) \\
    & \hspace{2.5cm}  + h^1(\sEnd(E_{\vec{a}})) - h^1(\sHom(E_{\vec{a}}, N_{C/\P^n})) - h^0(\sEnd(\cO(de)))\\
    = & h^1(\sEnd(E_{\vec{a}})) - h^1(\sHom(E_{\vec{a}}, N_{C/\P^n})).
\end{align*}
\end{proof}

\begin{corollary}\label{corhom} If $e=n$ or $a_i\le e+1$ for all $i$, then the codimension of $\Phi_{\vec{a}}$ in $\Hom (N_{C/\P^n}, \cO_{\P^1}(de))$ is $h^1(\sEnd(E_{\vec{a}}))$.
\end{corollary}
\begin{proof}
In this case, since $N_{C/\P^n}\cong \cO(e+2)^{e-1}\oplus \cO(e)^{n-e}$, we have $h^1(\sHom(E_{\vec{a}}, N_{C/\P^n})) = 0$.
\end{proof}

We can compute the dimension of the balanced splitting type, recovering a case of \cite{CR}, Corollary 2.2.

\begin{corollary}\label{phiissurjective} (\cite{CR}, Corollary 2.2) 
The kernel of the general $M\in \Hom (N_{C/\P^n}, \cO(de))$ is balanced.
\end{corollary}
\begin{proof}
For $E_{\vec{a}}$ the balanced splitting type, we have $a_i\le \left \lceil \frac{e(n-d+1)-2}{n-2} \right \rceil \le e+1$ and $|a_i-a_j|\le 1$ for all $i,j$, thus
\[ \codim \Phi_{\vec{a}} = h^1(\sEnd(E_{\vec{a}})) = 0. \]

Hence $\Phi_{\vec{a}}$ has dimension equal to $\dim \Hom (N_{C/\P^n}, \cO(de))$.
\end{proof}

\begin{theorem}\label{teoremasigmaa} The locus $\Sigma_{\vec{a}}$ is irreducible and smooth of codimension $h^1(\sEnd(E_{\vec{a}})) - h^1(\sHom(E_{\vec{a}}, N_{C/\P^n}))$ in $\Sigma$.
\end{theorem}
\begin{proof}
Let $\overline{\Sigma_{\vec{a}}}$ be the space of all polynomials $F\in H^0(\cI_C(d))$, not necessarily smooth along $C$, whose kernel of $\psi_F$ has splitting type $E_{\vec{a}}$. Consider the map $\beta :\overline{\Sigma_{\vec{a}}} \to \Phi_{\vec{a}}$, $\beta (F) = \psi _F$. By \cite{CR} Theorem 3.1, the map $\phi : H^0(\cI_C(d))\to \Hom (N_{C/\P^n},\cO(de))$ is surjective. It follows that $\beta$ is also surjective. In addition, by the short exact sequence 
\[ 0\longrightarrow H^0(\cI_{C/\P^n}^2(d))\longrightarrow H^0(\cI_{C/\P^n}(d))\overset{\phi}{\longrightarrow} \Hom (N_{C/\P^n}, \cO(de))\longrightarrow 0, \]
the fiber $\beta ^{-1}(M)$ is isomorphic to the linear system $H^0(\cI_C^2(d))$. Thus, $\beta ^{-1}(M)$ is smooth and irreducible of dimension $h^0(\cI^2_C(d))$. Then, by Proposition \ref{lemmaHom}, it follows that $\overline{\Sigma_{\vec{a}}}$ is smooth and irreducible of dimension
\[ h^0(\sHom(N_{C/\P^n}, \cO(de))) - h^1(\sEnd(E_{\vec{a}})) + h^1(\sHom(E_{\vec{a}}, N_{C/\P^n})) + h^0(\cI^2_C(d)). \]

Since $F$ being smooth along $C$ is an open condition in $\overline{\Sigma_{\vec{a}}}$, and $\Sigma_{\vec{a}}$ is not empty by Theorem \ref{existence}, it follows that $\Sigma_{\vec{a}}$ is an open dense subset of $\overline{\Sigma_{\vec{a}}}$. Therefore, $\Sigma_{\vec{a}}$ is irreducible and smooth of the same dimension.

The dimension of $\Sigma$ is $h^0(\cI_C(d))$, and by the sequence above,
\[ h^0(\cI^2_C(d)) = h^0(\cI_C(d)) - h^0(\sHom(N_{C/\P^n}, \cO(de))). \]

Hence, the codimension of $\Sigma _{\vec{a}}$ in $\Sigma$ is
\[ \codim (\Sigma_{\vec{a}}\subset \Sigma) = h^1(\sEnd(E_{\vec{a}})) - h^1(\sHom(E_{\vec{a}}, N_{C/\P^n})). \]
\end{proof}

By Corollary \ref{smoothexample}, when $d\ge4$ the general hypersurface of $\Sigma_{\vec{a}}$ is smooth.

\begin{corollary}
(char $K=0$) Let $d\ge 4$, and assume the base field $K$ has characteristic $0$. Let $S\Sigma_{\vec{a}}$ be the subspace of $\Sigma_{\vec{a}}$ of polynomials $F$ with $X$ smooth. Then $S\Sigma_{\vec{a}}$ is irreducible and smooth of codimension $h^1(\sEnd(E_{\vec{a}})) - h^1(\sHom(E_{\vec{a}}, N_{C/\P^n}))$ in $\Sigma$.
\end{corollary}

As in Corollary \ref{corhom}, we get the expected codimension when $e=n$ or all $a_i$ are smaller than $e+2$.

\begin{corollary} If $e=n$ or $a_i\le e+1$ for all $i$, then the codimension of $\Sigma_{\vec{a}}$ in $\Sigma$ is the expected $h^1(\sEnd(E_{\vec{a}}))$.
\end{corollary}

When $e < n$ and there exist terms $a_i = e+2$, the expected and the actual codimension differ. We can compute this difference as follows.

\begin{corollary}\label{difference} Let $z = |\{ i \ | \ a_i = e+2 \}|$ be the number of terms $\cO(e+2)$ in the splitting type $E_{\vec{a}}$. Then the codimension of $\Sigma_{\vec{a}}$ in $\Sigma$ is $h^1(\sEnd(E_{\vec{a}})) - (n-e)z$.
\end{corollary}
\begin{proof}
By Theorem \ref{teoremasigmaa} and Serre duality for $\P^1$,
\begin{align*}
 h^1(\sHom(E_{\vec{a}}, N_{C/\P^n})) & = h^1(\sHom(E_{\vec{a}}, \cO(e+2)^{e-1}\oplus \cO(e)^{n-e}))  \\ 
    & = h^0(\sHom(\cO, (\cO(-e-2)^{e-1}\oplus \cO(-e)^{n-e})\otimes (\oplus_{i=1}^{n-2} \cO(a_i))\otimes \cO(-2)))  \\
    & = \sum_{i=1}^{n-2} h^0(\P^1, \cO(-e-4+a_i)^{e-1}\oplus \cO(-e-2+a_i)^{n-e}) = (n-e)z.
\end{align*}
\end{proof}

In particular, the difference between the actual and the expected codimension can get arbitrarily large as $n$ grows.

\section{Quadric Hypersurfaces}

In this section, we study the case $d=2$. Let $X = V(F)$ be a degree 2 hypersurface in $\P^n$ containing the rational normal curve $C$ of degree $e$, and consider the exact sequence of the map $\phi$ described in Section 2:
\[ 0\longrightarrow H^0(\cI^2_C(2))\longrightarrow H^0(\cI_C(2))\overset{\phi}{\longrightarrow} \Hom (N_{C/\P^n}, \cO_{\P^1}(2e)). \tag{4} \]

Unlike when $d\ge 3$, the map $\phi$ is not surjective. Thus, the cokernel of an injection $E_{\vec{a}}\hookrightarrow N_{C/\P^n}$ may not be in the image of $\phi$, so we cannot repeat the arguments from Proposition \ref{lemmaHom} to compute the dimension of $\Sigma_{\vec{a}}$. Nevertheless, we can compute the dimension of the image of $\phi$.

\begin{proposition}\label{propdim} The dimension of $H^0(\cI_C^2(2))$ is $\frac{(n-e)(n-e+1)}{2}$. In particular, $\phi$ is injective when $e=n$.

\end{proposition}
\begin{proof}
If $G\in H^0(\cI_C^2(2))$ then $G$ is double along $C$. Since $G$ is a quadric, and the singular locus of a quadric is a linear space, $G$ must also be singular on all the $\P^e$ spanned by $C$. Thus, $G$ must be a combination of the generators $x_ix_j$, $e+1\le i,j\le n$ of $\cI^2_{\P^e}(2)$. That is,
    \[ G = \sum _{e+1\le i\le j\le n} \lambda_{ij}x_ix_j, \ \ \lambda_{ij}\in K. \]
    
    So we got to choose $\frac{(n-e)(n-e+1)}{2}$ coefficients to define $G$. Thus, $h^0(\cI_C^2(2)) = \frac{(n-e)(n-e+1)}{2}$.
\end{proof}

\begin{corollary} The image of $\phi$ has dimension $\frac{2ne + 2n -3e - e^2}{2}$.
\end{corollary}
\begin{proof}
From the exact sequence $(4)$, $\dim (\mathrm{im} \ \phi ) = h^0(\cI_C(2)) - h^0(\cI^2_C(2))$.
Thus, by Proposition \ref{propdim},
\[ \dim (\mathrm{im} \ \phi) = \dbinom{n+2}{2} - (2e+1) - \frac{(n-e)(n-e+1)}{2} = \frac{2ne + 2n -3e - e^2}{2}. \]
\end{proof}

\subsection{Splitting Types}

As in the case $d\ge 3$, we start by asking which splitting types arise as the normal bundle $N_{C/X}$. We obtain examples of quadrics smooth along $C$ for each possible splitting type.

The normal bundle sequence in the case $d=2$ is
\[ 0\longrightarrow N_{C/X}\longrightarrow \cO(e+2)^{e-1}\oplus \cO(e)^{n-e}\overset{\psi_F}{\longrightarrow}\cO(2e)\longrightarrow 0 \]
and $\deg N_{C/X} = e(n-1)-2$. Then, any splitting type for $N_{C/X}$ must have the form of a direct sum $\left ( \bigoplus_{i=1}^{e-1}\cO(e+2-a_i) \right )\oplus \left ( \bigoplus _{j=e+2}^{n}\cO(e-b_j)\right )$ with $a_i, b_j\ge 0$ and $\sum_{i=1}^{e-1}a_i+\sum_{j=e+2}^nb_j = e$. As in the case $d\ge 3$, all splitting types with at most $e-2$ terms of degree greater than $e$ are achieved. Splitting types with a term $\cO(e+2)^{e-1}$ come from hypersurfaces containing the $e$-plane spanned by $C$.

\begin{theorem}\label{quadricprop1}
\leavevmode
\begin{enumerate}[label=(\alph*)]
    \item For any given splitting type of the form
    \[ E = \left ( \bigoplus_{i=1}^{e-2}\cO(e+2-a_i) \right )\oplus \left ( \bigoplus _{j=e+1}^{n}\cO(e-b_j)\right ), \] with $a_i, b_j\ge 0$ and $\sum_{i=1}^{e-2}a_i + \sum_{j=e+1}^nb_j = e-2$, we produce a quadric $X=V(F)$, smooth along $C$, for which $N_{C/X}\cong E$.

    More explicitly, rearrange the $a_i$ in non-decreasing order $0 = a_y < a_{y+1}\le \cdots \le a_{e-2}$, and let $\beta_0 = 0$ and $\beta_i = a_{y+1} + \cdots + a_{y+i}$ for $i=1, \ldots , e-2-y$:
    
    For $e=n$, the quadric $X$ given by the polynomial
    \[ F = \sum_{i=0}^{n-2-y}Q_{\beta_i+1,\beta_i+2} =  Q_{1,2} + Q_{\beta_1 + 1, \beta_1 +2} + Q_{\beta_2 + 1, \beta_2 + 2} +  \cdots + Q_{\beta_{n-3-y}+1, \beta_{n-3-y}+2} + Q_{n-1,n}, \]
    where $Q_{ij} = x_ix_{j-1}-x_{i-1}x_j$, has normal bundle $N_{C/X}\cong E$.
    
    For $e<n$, let also $\gamma_n = e$ and $\gamma_{j} = e - b_n - b_{n-1} - \cdots - b_{j+1}$ for $e+1\le j\le n-1$. Then a quadric $X$ such that $N_{C/X}\cong E$ is given by
    \begin{align*}
        F = & \sum_{i=0}^{e-2-y}Q_{\beta_i+1,\beta_i+2} + \sum_{j=e+1}^n x_{\gamma_j}x_j\\
        = & (Q_{1,2} + Q_{\beta_1 + 1, \beta_1 +2} + \cdots + Q_{\beta_{e-2-y}+1, \beta_{e-2-y}+2}) + (x_{\gamma_{e+1}}x_{e+1} + \cdots + x_{\gamma_{n-1}}x_{n-1} + x_ex_n).
    \end{align*}
    
    \item If the splitting type of $N_{C/X}$ contains $e-1$ terms of degree greater than $e$, then it must be of the form
    \[ E' = \cO(e+2)^{e-1}\oplus \left ( \bigoplus _{j=e+2}^{n}\cO(e-b_j)\right ), \]
    with $b_j\ge 0$, $\sum_{j=e+2}^n b_j = e$ and $e<n-1$. In this case, $X$ contains the $e$-plane spanned by $C$. We produce a quadric $X$, smooth along $C$, for which $N_{C/X}\cong E'$.
    
    More explicity, let $\gamma_n = e$ and $\gamma_j = e-b_n-b_{n-1}-\cdots -b_{j+1}$ for $e+1\le j\le n-1$, and let $X$ be the quadric given by the polynomial
    \[ F = \sum_{j=e+1}^n x_{\gamma_j}x_j = x_0x_{e+1} + x_{b_{e+2}}x_{e+2} + \cdots + x_{\gamma_{n-1}}x_{n-1} + x_ex_n. \]
    Then $N_{C/X}\cong E'$.

\end{enumerate}
\end{theorem}

\begin{examplenotitalic}
Before we proceed to the proof of the theorem, let us use an example to better illustrate the idea of the proof, and to show how we are using the relations between the entries of $\psi_F$ to compute its kernel.

Let $n=e=5$. Consider polynomials $F$ of the form $F = \lambda_1Q_{1,2} + \lambda_2Q_{2,3} + \lambda_3Q_{3,4} + \lambda_4Q_{4,5}$. So $F$ induces the map $\psi_F: \cO(7)^{4}\to \cO(10)$,
\[ \psi_F = (\lambda_1s^3, \ \lambda_2s^2t, \ \lambda_3st^2, \ \lambda_4t^3). \]

When $\lambda_i = 1$, $i=1,\dots ,4$, $\psi_F = (s^3, \ s^2t, \ st^2, \ t^3)$ and we get the following degree $1$ relations between consecutive entries of $\psi_F$, which we call \textit{column relations} of $\psi_F$:
\begin{align*}
    & t(s^3) - s(s^2t) = 0\\
    & t(s^2t) - s(st^2) = 0\\
    & t(st^2) - s(t^3) = 0\\
\end{align*}
Writing the coefficients of these relations as column vectors, we get the matrix $K$
\[ K = \begin{pmatrix}
t & 0 & 0\\
-s & t & 0\\
0 & -s & t\\
0 & 0 & -s
\end{pmatrix}, \]
which defines a map $\cO(6)^3\overset{K}{\to} \cO(7)^4$. The column relations imply that the image of $K$ is contained in the kernel of $\psi_F$. As $K$ has rank 3, the map is injective and it coincides with the kernel. Hence, $N_{C/X}\cong \cO(6)^3$.

Note that the splitting type of $N_{C/X}$ is determined by the degrees of the column relations of $\psi_F$. So, if we wish to get $N_{C/X}\cong \cO(5)\oplus \cO(6)\oplus \cO(7)$, we need relations of degree $0$, $1$ and $2$. Consecutive nonzero entries give degree $1$ relations. Similarly, entries separated by one entry give a degree $2$ relation. Then, let $\lambda_3 = 0$, so $\psi_F = (s^3, \ s^2t, \ 0, \ t^3)$. It satisfies the column relations
\begin{align*}
    & t(s^3) - s(s^2t) = 0\\
    & t^2(s^2t) - s^2(t^3) = 0
\end{align*}
and similarly, we define the matrix
\[ K = \begin{pmatrix}
t & 0 & 0\\
-s & 0 & t^2\\
0 & 1 & 0\\
0 & 0 & -s^2
\end{pmatrix}, \]
and obtain that $N_{C/X}\cong \cO(6)\oplus \cO(7)\oplus \cO(5)\overset{K}{\to} \cO(7)^4$ is the kernel of $\psi_F$.

Finally, the splitting type $\cO(4)\oplus \cO(7)^2$ must come from a $\psi_F$ that has a column relation of degree $3$. We get it by letting two consecutive entries of $\psi_F$ be $0$, that is, let $\psi_F = (s^3, \ 0, \ 0, \ t^3)$.

In summary, for every term $\cO(n+2-a_i)$ of the splitting type, we need a corresponding column relation of degree $a_i$. We can get it by letting $a_i-1$ consecutive terms of $\psi_F$ be $0$. By doing it for all $i$, we get all needed relations.   
\end{examplenotitalic}

\medskip
\begin{proofofthm4.3}
(a) Let first $e=n$. In this case, the possible splitting types are $E = \bigoplus_{i=1}^{n-2}\cO(n+2-a_i)$ with $a_i\ge 0$ and $\sum_{i=1}^{n-2}a_i = n-2$. We look at polynomials $F$ of the form $F = \sum_{i=1}^{n-1} \lambda_iQ_{i,i+1}$ with $\lambda_i\in \{ 0, 1 \}$ for all $i$. Then, $F$ induces the map $\psi_F: \cO(n+2)^{n-1}\to \cO(2n)$,
\[ \psi_F = (\lambda_1 s^{n-2}, \ \lambda_2 s^{n-3}t, \cdots , \lambda_i s^{n-1-i}t^{i-1}, \cdots , \lambda_{n-1}t^{n-2}). \]

As in the examples above, we want to find a $\psi_F$ with column relations of degrees $a_i$. We can get it by letting $a_i-1$ consecutive entries of $\psi_F$ be $0$ for every $a_i\ge 1$. So, write the $a_i$ in non-decreasing order $0 = a_y < a_{y+1}\le \cdots \le a_{n-2}$, and let $\beta_0 = 0$ and $\beta_i = a_{y+1} + \cdots + a_{y+i}$ for $i=1, \ldots , n-2-y$. So,
\[ \psi_F = (s^{n-2}, \ 0, \dots , 0, \ s^{n-2-\beta_1}t^{\beta_1},\ 0, \dots , 0,\  s^{n-2-\beta_2}t^{\beta_2}, \ \ \cdots \ \ , \ s^{n-2-\beta_{n-3}}t^{\beta_{n-3}}, \ 0, \dots , 0 , \ t^{n-2}). \]

Notice that $\psi_F$ is induced by the polynomial
\[ F = \sum_{i=0}^{n-2-y}Q_{\beta_i+1,\beta_i+2} =  Q_{1,2} + Q_{\beta_1 + 1, \beta_1 +2} + Q_{\beta_2 + 1, \beta_2 + 2} \cdots + Q_{\beta_{n-3-y}+1, \beta_{n-3-y}+2} + Q_{n-1,n}. \]

The entries of $\psi_F$ satisfy the degree $a_{i+1}$ relations
\begin{align*}
 & t^{\beta_{i+1}-\beta_i}(s^{n-2-\beta_i}t^{\beta_i}) - s^{\beta_{i+1}-\beta_i}(s^{n-2-\beta_{i+1}}t^{\beta_{i+1}})\\
 & = t^{a_{y+i+1}}(s^{n-2-\beta_i}t^{\beta_i}) - s^{a_{y+i+1}}(s^{n-2-\beta_{i+1}}t^{\beta_{i+1}}) \ \ \text{ for } \ \ 0\le i\le n-3-y,
\end{align*}
in addition to the $y$ degree $0$ relations corresponding to the zero entries of $\psi_F$.

Since each relation involves either a different pair of nonzero columns, or a single zero column, they are linearly independent, that is, the matrix $K$ whose columns are the coefficients of the relations has maximal rank. Therefore, $K$ gives the kernel of $\psi_F$, hence 
\[ N_{C/X}\cong \cO(n+2)^y\oplus \left (\bigoplus_{i=1}^{n-2-y}\cO(n+2-a_{y+i})\right ) = \bigoplus_{i=1}^{n-2}\cO(n+2-a_i). \]

Finally, we show that these $F$ are smooth along the rational normal curve. Consider the partial derivatives of $F$
\[ \frac{\partial F}{\partial x_0} = -x_2, \ \ \frac{\partial F}{\partial x_2} = 2\lambda_2x_2 - x_0 - \lambda_3x_4, \ \ \frac{\partial F}{\partial x_{n-2}} = 2x_{n-2} - x_n - \lambda_{n-3}x_{n-4}. \]
If $F$ is singular at a point $P = (s^n: s^{n-1}t: \cdots : t^n)\in C$, then $\frac{\partial F}{\partial x_0} = 0$ at $P$ implies $s=0$ or $t=0$. If $s=0$, then $\frac{\partial F}{\partial x_{n-2}} = 0$ implies $t=0$; and if $t=0$ then $\frac{\partial F}{\partial x_2} = 0$ implies $s=0$.  Hence, $F$ is smooth along $C$.

\medskip

Now, let $e<n$. We want to obtain all the possible splitting types $E = \left ( \bigoplus_{i=1}^{e-2}\cO(e+2-a_i) \right )\oplus \left ( \bigoplus _{j=e+1}^{n}\cO(e-b_j)\right )$. Consider polynomials $F$ of the form $F = \sum_{i=1}^{n-1}\lambda_iQ_{i,i+1} + \sum_{j=e+1}^{n}L_jx_j$ with $\lambda_i\in \{0,1\}$ and $L_j$ linear forms. Then $F$ induces the map $\psi_F: \cO(e+2)^{e-1}\oplus \cO(e)^{n-e}\to \cO(2e)$,
\[ \psi_F = (\lambda_1 s^{n-2}, \ \lambda_2 s^{n-3}t, \cdots , \lambda_i s^{n-1-i}t^{i-1}, \cdots , \lambda_{n-1}t^{n-2}; \ L_{e+1}|_C, L_{e+2}|_C, \cdots , L_n|_C). \]

We separate $\psi_F$ into two parts, one corresponding to the $\lambda_is^{n-1-i}t^{i-1}$ and the other to the $L_j|_C$. The idea is to use the first part to obtain all the terms $\cO(e+2-a_i)$ of $E$ as in the case $e=n$, and to use the second part for the $\cO(e-b_j)$ terms. 

Consider the $a_i$ in non-decreasing order $0 = a_y < a_{y+1}\le \cdots \le a_{e-2}$, and let $\beta_0 = 0$ and $\beta_i = a_{y+1} + \cdots + a_{y+i}$ for $i=1, \ldots , e-2-y$. Let also $\gamma_n = e$ and $\gamma_{j} = e - b_n - b_{n-1} - \cdots - b_{j+1}$ for $e+1\le j\le n-1$. Consider the map
\begin{multline*}
    \psi_F = (s^{n-2}, \ 0, \dots , 0, \ s^{n-2-\beta_1}t^{\beta_1},\ 0, \dots , 0,\  s^{e-2-\beta_2}t^{\beta_2}, \ \ \cdots \ \ , \ s^{e-2-\beta_{e-2-y}}t^{\beta_{e-2-y}}, \ 0, \dots 0 ;\\
     \ \ s^{e-\gamma_{e+1}}t^{\gamma_{e+1}}, \ s^{e-\gamma_{e+2}}t^{\gamma_{e+2}}, \dots , s^{e-\gamma_{n-1}}t^{\gamma_{n-1}}, \ t^e ).
\end{multline*}

Notice $\psi_F$ is induced by the polynomial
\begin{align*}
    F = & \sum_{i=0}^{e-2-y}Q_{\beta_i+1,\beta_i+2} + \sum_{j=e+1}^n x_{\gamma_j}x_j\\
    = & (Q_{1,2} + Q_{\beta_1 + 1, \beta_1 +2} + \cdots + Q_{\beta_{e-2-y}+1, \beta_{e-2-y}+2}) + (x_{\gamma_{e+1}}x_{e+1} + \cdots + x_{\gamma_{n-1}}x_{n-1} + x_ex_n).
\end{align*}

We obtain the following relations between the entries of the first part of $\psi_F$:
\begin{align*}
 & t^{\beta_{i+1}-\beta_i}(s^{e-2-\beta_i}t^{\beta_i}) - s^{\beta_{i+1}-\beta_i}(s^{e-2-\beta_{i+1}}t^{\beta_{i+1}}) = 0\\
 & \Leftrightarrow t^{a_{y+i+1}}(s^{e-2-\beta_i}t^{\beta_i}) - s^{a_{y+i+1}}(s^{e-2-\beta_{i+1}}t^{\beta_{i+1}}) = 0 \ \ \text{ for } \ \ 0\le i\le e-3,
\end{align*}
in addition to the $y$ degree $0$ relations corresponding to the zero entries.

And between every two consecutive entries of the second part of $\psi_F$:
\begin{align*}
 & t^{\gamma_{j+1}-\gamma_j}(s^{e-\gamma_j}t^{\gamma_j}) - s^{\gamma_{j+1}-\gamma_j}(s^{e-\gamma_{j+1}}t^{\gamma_{j+1}}) = 0\\
 & \Leftrightarrow t^{b_{j+1}}(s^{e-\gamma_j}t^{\gamma_j}) - s^{b_{j+1}}(s^{e-\gamma_{j+1}}t^{\gamma_{j+1}}) = 0 \ \ \text{ for } \ \ e+1\le j\le n-1.
\end{align*}

Note that $\gamma_{e+1}-\beta_{e-2-y} = e - \left (\sum_{i=1}^{e-2} a_i + \sum_{j=e+2}^n b_j\right ) = e - (e-2-b_{e+1}) = b_{e+1}+2$. We also get a relation between the first and second part of $\psi_F$:
\begin{align*}
 & t^{\gamma_{e+1}-\beta_{e-2-y}}(s^{e-2-\beta_{e-2-y}}t^{\beta_{e-2-y}}) - s^{\gamma_{e+1}-\beta_{e-2-y}-2}(s^{e-\gamma_{e+1}}t^{\gamma_{e+1}}) = 0 \\
 & \Leftrightarrow t^{b_{e+1}+2}(s^{e-2-\beta_{e-2-y}}t^{\beta_{e-2-y}}) - s^{b_{e+1}}(s^{e-\gamma_{e+1}}t^{\gamma_{e+1}}) = 0.
\end{align*}

The matrix $K$ defined by the coefficients of these relations defines a map
\[ K: \left ( \bigoplus_{i=1}^{e-2}\cO(e+2-a_i) \right )\oplus \left ( \bigoplus _{j=e+1}^{n}\cO(e-b_j)\right )\to \cO(e+2)^{e-1}\oplus \cO(e)^{n-e} \]
that factors through the kernel $N_{C/X}$ of $\psi_F$. Since the relations involve different pairs of nonzero columns, or single zero columns, $K$ has maximum rank, so the map is injective, and it follows that
\[ N_{C/X}\cong \left ( \bigoplus_{i=1}^{e-2}\cO(e+2-a_i) \right )\oplus \left ( \bigoplus _{j=e+1}^{n}\cO(e-b_j)\right ). \]

To conclude, we check that $F$ is smooth along the curve. Consider the partial derivatives restricted to $C$ (that is, take $Q_{ij} = 0$ and $x_l = 0$ for $l\ge e+1$)
\[ \frac{\partial F}{\partial x_2} = -x_0 + 2\lambda_2x_2 - \lambda_3x_4, \ \ \frac{\partial F}{\partial x_{n}} = x_e. \]
If $P = (s^e : s^{e-1}t : \cdots : t^e : 0 : \cdots : 0)\in C$ is a singular point of $X$, then $\frac{\partial F}{\partial x_{n}} = x_e = t^e = 0$ implies $t=0$, thus $\frac{\partial F}{\partial x_2} = -x_0 + 2\lambda_2x_2 - \lambda_3x_4 = -s^e = 0$, so $s=0$. Hence, $X$ is smooth along $C$.

\medskip
(b) As we have seen in part (a), the summands of degree greater than $e$ come from column relations involving only entries of the first part of the map $\psi_F$, as otherwise it defines a nonzero map to an $\cO(e)$ term. Since the first part has $e-1$ entries, we can only obtain these $e-1$ relations if the first part is zero, that is, if $\psi_F$ has the form
\[ \psi_F = (0, \ 0, \hdots ,\ 0; \ L_{e+1}|_C, \ L_{e+2}|_C, \hdots , \ L_n|_C). \]
The column relations $C_i = 0$ for $1\le i\le e-1$ induce the summand $\cO(e+2)^{e-1}$ of $N_{C/X}$.

Now, from Section 2, equation (1), notice that any map $\psi_F$ has the form $(C_1, \cdots , C_{e-1}; L_{e+1}|_C, \cdots , L_n|_C)$, and that the $C_i$ in the first part do not depend on the $L_j|_C$, that is, they are the same $C_1, \hdots , C_{e-1}$ we get in the case $e=n$ with $F$ the corresponding polynomial without the $L_jx_j$ terms. Thus, since $\phi$ is injective when $e=n$ by Proposition \ref{propdim}, the only polynomials $F$ with $C_1 = \cdots = C_{e-1} = 0$ are the polynomials of the form $F = \sum _{j=e+1}^n L_jx_j$. Therefore, $X$ contains the $e$-plane spanned by $C$.

If $e = n-1$, then we would have $N_{C/X}\cong \cO(e+2)^{e-1}$, which cannot happen for degree reasons. So, we assume $e < n-1$, and let 
\[ F = \sum_{j=e+1}^n x_{\gamma_j}x_j = x_0x_{e+1} + x_{b_{e+2}}x_{e+2} + \cdots + x_{\gamma_{n-1}}x_{n-1} + x_ex_n, \]
with $\gamma_n = e$ and $\gamma_j = e-b_n-b_{n-1}-\cdots -b_{j+1}$ for $e+1\le j\le n-1$. Then the map $\psi_F: \cO(e+2)^{e-1}\oplus \cO(e)^{n-e}\to \cO(2e)$ is
\begin{align*}
\psi_F & = (0, \ 0, \cdots ,\ 0; \ x_0|_C, \ x_{b_{e+2}}|_C, \cdots , x_{\gamma_{n-1}}|_C , \ x_n|_C)\\
 & = (0, \ 0, \cdots ,\ 0; \ s^e, \ s^{e-\gamma_{e+2}}t^{\gamma_{e+2}}, \cdots , s^{e-\gamma_{n-1}}t^{\gamma_{n-1}} , \ t^e).
\end{align*}

We obtain the following relations between the columns of the second part of $\psi_F$:
\begin{align*}
 & t^{\gamma_{j+1}-\gamma_j}(s^{e-\gamma_j}t^{\gamma_j}) - s^{\gamma_{j+1}-\gamma_j}(s^{e-\gamma_{j+1}}t^{\gamma_{j+1}}) = 0\\
 & \Leftrightarrow t^{b_{j+1}}(s^{e-\gamma_j}t^{\gamma_j}) - s^{b_{j+1}}(s^{e-\gamma_{j+1}}t^{\gamma_{j+1}}) = 0 \ \ \text{ for } \ \ e+1\le j\le n-1.
\end{align*}

Thus, these relations define a map
\[ K: \cO(e+2)^{e-1}\oplus \left ( \bigoplus _{j=e+2}^{n}\cO(e-b_j)\right )\to \cO(e+2)^{e-1}\oplus \cO(e)^{n-e} \]
that factors through the kernel of $\psi_F$. As the relations are independent, $K$ has maximum rank, and we have
\[ N_{C/X}\cong \cO(e+2)^{e-1}\oplus \left ( \bigoplus _{j=e+2}^{n}\cO(e-b_j)\right ). \]

Finally, taking partial derivatives and restricting to $C$, we have
\[ \frac{\partial F}{\partial x_{e+1}} = s^e \ \ \text{and} \ \ \frac{\partial F}{\partial x_n} = t^e, \]
so if a point $P=(s^e : \cdots : t^e : 0: \cdots :0)$ in $C$ is a singular point of $X$, we would have $s=t=0$. Therefore, $X$ is smooth along $C$.
\end{proofofthm4.3}

\medskip
A natural question that arises is whether we can obtain all splitting types from smooth quadric hypersurfaces $X$. The construction above produce quadrics smooth along $C$, but not necessarily smooth. We can, however, use the quadratic form matrix of $F$ to reduce the dimension of the singular locus of $X$.

\begin{theorem}\label{corankexample}
\leavevmode
\begin{enumerate}[label=(\alph*)]
    \item For every splitting type
    \[ E = \left ( \bigoplus_{i=1}^{e-2}\cO(e+2-a_i) \right )\oplus \left ( \bigoplus _{j=e+1}^{n}\cO(e-b_j)\right ), \]
    with $a_i, b_j\ge 0$ and $\sum_{i=1}^{e-2}a_i + \sum_{j=e+1}^nb_j = e-2$, we obtain an example of quadric $X$ of corank at most $\sum _{a_i\ge 4}(a_i-3)$ with $N_{C/X}\cong E$. In particular, if $a_i\le 3$ for all $i$, there exists a smooth quadric $X$ with $N_{C/X}\cong E$.
    \medskip
    \item For splitting types of the form
    \[ E' = \cO(e+2)^{e-1}\oplus \left ( \bigoplus _{j=e+2}^{n}\cO(e-b_j)\right ), \]
    with $b_j\ge 0$ and $\sum _{j=e+2}^n b_j = e$, let $w = |\{ j \ | \ b_j = 0 \}|$ be the number of terms of degree $e$. Then we obtain a quadric $X$ of corank $e+1 - \mathrm{min}\{e+1, n-e-w\}$ with $N_{C/X}\cong E'$. In particular, if $e+1\le n-e-w$, then there exists a smooth quadric $X$ with $N_{C/X}\cong E'$.
\end{enumerate}
\end{theorem}
\begin{proof} (a) We first refer to the proof of Theorem \ref{quadricprop1}, as we will use the quadric $F$ constructed there. We divide the proof into the cases $e=n$ and $e<n$.

Let $e=n$. Rearrange the $a_i$ in non-decreasing order $0 = a_y < a_{y+1}\le \cdots \le a_{n-2}$, and follow the construction of $F$ in the proof of Theorem \ref{quadricprop1}. For $\beta_0 = 0$ and $\beta_i = a_{y+1} + \cdots + a_{y+i}$ for $i=1, \ldots , n-2-y$, we obtain the polynomial
\[ F = \sum_{i=0}^{n-2-y}Q_{\beta_i+1,\beta_i+2} =  Q_{1,2} + Q_{\beta_1 + 1, \beta_1 +2} + Q_{\beta_2 + 1, \beta_2 + 2} \cdots + Q_{\beta_{n-3-y}+1, \beta_{n-3-y}+2} + Q_{n-1,n}. \]
corresponding to the map
\[ \psi_F = (s^{n-2}, \ 0, \dots , 0, \ s^{n-2-\beta_1}t^{\beta_1},\ 0, \dots , 0,\  s^{n-2-\beta_2}t^{\beta_2}, \ \ \cdots \ \ , \ s^{n-2-\beta_{n-3}}t^{\beta_{n-3}}, \ 0, \dots , 0 , \ t^{n-2}), \]
whose kernel is $N_{C/X}\cong \bigoplus_{i=1}^{n-2}\cO(n+2-a_i)$.

We look at the quadratic form matrix of $F$. It is the $(n+1)\times (n+1)$ symmetric matrix $Q = (c_{i,j})_{i,j=0}^n$ with entries $c_{\beta_{i}+1,\beta_{i}+1} = 1$, $c_{\beta_{i},\beta_{i}+2} = -\frac{1}{2}$, $c_{\beta_{i}+2,\beta_{i}} = -\frac{1}{2}$, $0\le i\le n-2-y$, and zero elsewhere. It has 1's and 0's in the diagonal with $a_i-1$ consecutive 0's for each $a_i$, in increasing order of $i$. For example, for $n=5$ and $N_{C/X}\cong \cO(5)\oplus \cO(6)\oplus \cO(7)$, we have $a_1=0$, $a_2=1$ and $a_3=2$, and we construct $\psi_F = (s^3,\ s^2t,\ 0,\ t^3)$, that corresponds to the matrix
\[ Q = \left (\begin{smallmatrix}
0 & 0 & -\frac{1}{2} & 0 & 0 & 0 \\
0 & 1 & 0 & -\frac{1}{2} & 0 & 0 \\
-\frac{1}{2} & 0 & 1 & 0 & 0 & 0 \\
0 & -\frac{1}{2} & 0 & 0 & 0 & -\frac{1}{2} \\
0 & 0 & 0 & 0 & 1 & 0 \\
0 & 0 & 0 & -\frac{1}{2} & 0 & 0 
\end{smallmatrix}\right ). \]

Notice that, for each $a_i\ge 4$, we have a $(a_i-3)\times (a_i-3)$ zero diagonal block in $Q$. Thus, the rank of $F$ drops by at least $a_i-3$ for each relation of degree $a_i\ge 4$. The nonzero diagonal blocks might be singular, further reducing the rank of $F$. To prove the theorem, we will show we can replace the nonzero blocks by nonsingular ones without changing the splitting type of $N_{C/X}$.

First, notice that the first diagonal block of $Q$ corresponds to the $a_i\le 2$. This block is followed by $(a_i-3)\times (a_i-3)$ zero diagonal blocks alternating with blocks
\[ B = \begin{pmatrix}
0 & 0 & -\frac{1}{2}\\
0 & 1 & 0\\
-\frac{1}{2} & 0 & 0
\end{pmatrix}. \]

The blocks $B$ are already nonsingular, so we do not need to replace them. We might need to replace the first block. When we replace it by a new block, we change $F$ to an $F'$, and we need to check if the new $\psi_{F'}$ induces the same bundle $N_{C/X}$. We do this by checking that $\psi_{F'}$ has column relations of the same degree as of $\psi_F$.

Let us analyze the first block closer. It has size $(m+1)\times (m+1)$, for $m = \sum_{a_i\le 2}a_i + 2$. Its diagonal is formed by 0, a sequence of consecutive 1's, and then alternating 1's and 0's:
\[ B_1 = \left ( \begin{smallmatrix}
0 &  & -\frac{1}{2} &  &  &  &  &  &  &  & \\
 & 1 &  & -\frac{1}{2} &  &  &  &  &  &  & \\
-\frac{1}{2} &  & 1 &  &  &  &  &  &  &  & \\
 & -\frac{1}{2} &  & ... &  & -\frac{1}{2} &  &  &  &  & \\
 &  &  &  & 1 &  &  &  &  &  & \\
 &  &  & -\frac{1}{2} &  & 0 &  & -\frac{1}{2} &  &  & \\
 &  &  &  &  &  & 1 &  &  &  & \\
 &  &  &  &  & -\frac{1}{2} &  & ... &  &  & \\
 &  &  &  &  &  &  &  & 0 &  & -\frac{1}{2}\\
 &  &  &  &  &  &  &  &  & 1 & \\
 &  &  &  &  &  &  &  & -\frac{1}{2} &  & 0
\end{smallmatrix}\right ). \]

Let $l = |\{ i \ | \ a_i = 1 \}| + 1$ be the number of consecutive 1's in the diagonal. We divide into the cases $l=1, 2, 3$ and $l\ge 4$.

For the case $l=1$, the diagonal of $B_1$ alternates between $0$ and $1$, thus $m$ must be even. If $4\nmid m$, Gauss-Jordan elimination shows that the original matrix is nonsingular and we do not change it. If $4\mid m$, it may be singular, and we define $F' = F + Q_{\frac{m}{2}, \frac{m}{2}+1} + Q_{\frac{m}{2}-1, \frac{m}{2}+1}$. This change will replace the three middle entries of $\psi_F$:
\[ \psi_F = (s^{n-2}, \ 0, \ s^{n-4}t^2, \ 0, \cdots , s^{n-\frac{m}{2}}t^{\frac{m}{2}-2}, \ 0, \ s^{n-\frac{m}{2}-2}t^{\frac{m}{2}}, \cdots ,\ 0, \ s^{n-m}t^{m-2} \cdots ) \]
by
\begin{multline*}
    \psi_{F'} = (s^{n-2}, \ 0, \ s^{n-4}t^2, \ 0, \cdots , s^{n-\frac{m}{2}}t^{\frac{m}{2}-2} + s^{n-\frac{m}{2}-2}t^{\frac{m}{2}}, \ s^{n-\frac{m}{2}-1}t^{\frac{m}{2}-1}, \ s^{n-\frac{m}{2}-2}t^{\frac{m}{2}} + s^{n-\frac{m}{2}}t^{\frac{m}{2}-2}, \cdots \\
    \cdots ,\ 0, \ s^{n-m}t^{m-2} \cdots ),
\end{multline*}
and instead of the original column relations
\begin{align*}
    & t^2(s^{n-\frac{m}{2}}t^{\frac{m}{2}-2}) - s^2(s^{n-\frac{m}{2}-2}t^{\frac{m}{2}}) = 0\\
    & 1\cdot (s^{n-\frac{m}{2}}t^{\frac{m}{2}-2}) - 1\cdot 0 = 0,
\end{align*}
we get
\begin{align*}
    & (st)(s^{n-\frac{m}{2}}t^{\frac{m}{2}-2} + s^{n-\frac{m}{2}-2}t^{\frac{m}{2}}) - (s^2+t^2)(s^{n-\frac{m}{2}-1}t^{\frac{m}{2}-1}) = 0\\
    & 1\cdot (s^{n-\frac{m}{2}}t^{\frac{m}{2}-2} + s^{n-\frac{m}{2}-2}t^{\frac{m}{2}}) - 1\cdot (s^{n-\frac{m}{2}-2}t^{\frac{m}{2}} + s^{n-\frac{m}{2}}t^{\frac{m}{2}-2}) = 0,
\end{align*}
thus preserving the degrees of the column relations.

For $l=2$ or $l\ge 4$, Gauss-Jordan elimination shows that $B_1$ is nonsingular, and we do not replace it.

For $l=3$, consider the cases $m=4$ and $m>4$. If $m>4$, switch one of the 0's to the third row:
\[ \text{replace } \left ( \begin{smallmatrix}
0 &  & -\frac{1}{2} &  &  &  &  & &\\
 & 1 &  & -\frac{1}{2} &  &  &  & &\\
-\frac{1}{2} &  & 1 &  & -\frac{1}{2} &  &  & &\\
 & -\frac{1}{2} &  & 1 &  &  &  & &\\
 &  & -\frac{1}{2} &  & 0 &  & -\frac{1}{2} & &\\
 &  &  &  &  & 1 &  & &\\
 &  &  &  & -\frac{1}{2} &  & 0 & & -\frac{1}{2}\\
 &  &  &  &  &  &  & 1 &\\
 &  &  &  &  &  & -\frac{1}{2} &  &...
\end{smallmatrix}\right )\text{ by } 
\left ( \begin{smallmatrix}
0 &  & -\frac{1}{2} &  &  &  &  & &\\
 & 1 &  &  &  &  &  & &\\
-\frac{1}{2} &  & 0 &  & -\frac{1}{2} &  &  & &\\
 &  &  & 1 &  & -\frac{1}{2} &  & &\\
 &  & -\frac{1}{2} &  & 1 &  & -\frac{1}{2} & &\\
 &  &  & -\frac{1}{2} &  & 1 &  & &\\
 &  &  &  & -\frac{1}{2} &  & 0 & & -\frac{1}{2}\\
 &  &  &  &  &  &  & 1 &\\
 &  &  &  &  &  & -\frac{1}{2} &  &...
\end{smallmatrix}\right ).\]

The second matrix is nonsingular, as one can also see by Gauss-Jordan elimination. In terms of the polynomial, this corresponds to taking $F' = F - Q_{2,3} + Q_{4,5}$, and it replaces
\[ \psi_F = ( s^{n-2},\ s^{n-3}t, \ s^{n-4}t^2, \ 0, \ s^{n-6}t^4, \ 0, \ s^{n-8}t^6, \cdots ) \]
by
\[ \psi_F' = ( s^{n-2},\ 0, \ s^{n-4}t^2, \ s^{n-5}t^3, \ s^{n-6}t^4, \ 0, \ s^{n-8}t^6, \cdots ). \]
Note that this only changes the position of one degree $2$ relation, that is, instead of the column relations
\begin{align*}
    & t(s^{n-2}) - s(s^{n-3}t) = 0\\
    & t(s^{n-3}t) - s(s^{n-4}t^2) = 0\\
    & t^2(s^{n-4}t^2) - s^2(s^{n-6}t^4) = 0,
\end{align*}
we get the relations
\begin{align*}
    & t^2(s^{n-2}) - s^2(s^{n-4}t^2) = 0\\
    & t(s^{n-4}t^2) - s(s^{n-5}t^3) = 0\\
    & t(s^{n-5}t^3) - s(s^{n-6}t^4) = 0.
\end{align*}

If $m=4$, the first block of the original matrix $Q$ is
\[ \left (\begin{smallmatrix}
0 &  & -\frac{1}{2} &  & \\
 & 1 &  & -\frac{1}{2} & \\
-\frac{1}{2} &  & 1 &  & -\frac{1}{2}\\
 & -\frac{1}{2} &  & 1 & \\
 &  & -\frac{1}{2} &  & 0
\end{smallmatrix}\right ). \]

We replace it by the nonsingular
\[ \left (\begin{smallmatrix}
0 &  & -\frac{1}{2} & -\frac{1}{2} & -\frac{1}{2}\\
 & 1 & \frac{1}{2} & 0 & \\
-\frac{1}{2} & \frac{1}{2} & 1 &  & -\frac{1}{2}\\
-\frac{1}{2} & 0 &  & 1 & \\
-\frac{1}{2} &  & -\frac{1}{2} &  & 0
\end{smallmatrix}\right ), \]
corresponding to the polynomial $F' = F + Q_{1,3} + Q_{1,4}$. The original map $\psi_F$ is
\[ \psi_F = ( s^{n-2},\ s^{n-3}t, \ s^{n-4}t^4, \ 0, \cdots ) \]
and satisfies the degree $1$ column relations
\begin{align*}
    & t(s^{n-2}) - s(s^{n-3}t) = 0 \\
    & t(s^{n-3}t) - s(s^{n-4}t^2) = 0.
\end{align*}
The new map is
\[ \psi_{F'} = (s^{n-2} + s^{n-3}t + s^{n-4}t^2, \ s^{n-2}+2s^{n-3}t, \ s^{n-2} + s^{n-4}t^2, \cdots) \]
and satisfies the degree $1$ relations
\begin{align*}
    & (-s+3t)(s^{n-2} + s^{n-3}t + s^{n-4}t^2) + (s-t)(s^{n-2}+2s^{n-3}t) - (3t)(s^{n-2} + s^{n-4}t^2) = 0 \\
    & (-s-2t)(s^{n-2} + s^{n-3}t + s^{n-4}t^2) + t(s^{n-2}+2s^{n-3}t) + (s+2t)(s^{n-2} + s^{n-4}t^2) = 0,
\end{align*}
confirming that replacing $F$ by $F'$ does not change the degrees of the column relations.

So far, we have shown that we can replace the first block of the matrix of $F$ by a nonsingular block with column relations in $\psi_F$ of the same degrees. We now need to check the column relations between the first and the other blocks. It is enough to check the relation between the first and the second nonzero block.

Except for the case $l = 3$ and $m = 4$, the last entry of $\psi_F$ corresponding to the first block remains the same, and so do the column relations between the first and the second blocks. When $l = 3$ and $m = 4$, the map $\psi_F$ looks like
\[ \psi_{F} = (s^{n-2}, \ s^{n-3}t, \ s^{n-4}t^2, \ 0, \cdots , 0, \ s^{n-2-b}t^b, \cdots), \]
and the relation between the first and second blocks is the degree $b-2$ relation
\[ t^{b-2}(s^{n-4}t^2) - s^{b-2}(s^{n-2-b}t^b) = 0. \]

It gets replaced by
\[ \psi_{F'} = (s^{n-2} + s^{n-3}t + s^{n-4}t^2, \ s^{n-2}+2s^{n-3}t, \ s^{n-2} + s^{n-4}t^2, \ 0, \cdots , 0, \ s^{n-2-b}t^b, \cdots), \]
and we have the relation between the first and the second blocks
\[ t^{b-2}(2(s^{n-2} + s^{n-3}t + s^{n-4}t^2) - (s^{n-2}+2s^{n-3}t) - (s^{n-2} + s^{n-4}t^2)) - s^{b-2}(s^{n-2-b}t^b) = 0, \]
also of degree $b-2$.

Therefore, the new matrix we obtain by replacing the first block induces column relations of the same degree as in the original matrix. Hence, we still get $N_{C/X}\cong \bigoplus _{i=1}^{n-2}\cO(n+2-a_i)$. The rank of the new matrix is only decreased by the $(a_i-3)\times (a_i-3)$ zero diagonal blocks, thus it has corank $\sum _{a_i\ge 4}(a_i-3)$.

\medskip

Now, let $e<n$. The $F$ constructed in Theorem \ref{quadricprop1} is $F = \sum_{i=0}^{e-2-y}Q_{\beta_i+1,\beta_i+2} + \sum_{j=e+1}^n x_{\gamma_j}x_j$. The quadratic form matrix of $F$ can be seen as a four blocks symmetric matrix
\[ M = \begin{pmatrix}
Q & A \\
A^t & 0
\end{pmatrix} \]
where $Q$ is the $(e+1)\times (e+1)$ matrix corresponding to the quadratric form $\sum_{i=0}^{e-2-y}Q_{\beta_i+1,\beta_i+2}$ in $\P^e = V(x_{e+1}, \cdots , x_n)$. Thus, by applying the case $e=n$ to the matrix $Q$, we may assume $Q$ to be a corank $\sum _{a_i\ge 4}(a_i-3)$ matrix. In other words, we are able to replace $Q$ by a corank $\sum _{a_i\ge 4}(a_i-3)$ matrix with the same normal bundle $N_{C/X}$.

The $0$ block of $M$ corresponds to the terms $x_lx_j$ of $F$ with $l,j\ge e+1$. These terms are $0$ when restricted to $C$, and therefore do not interfere with the map $\psi_F$. Thus, we can replace the $0$ block by any matrix $L$ without changing $\psi_F$ and hence preserving $N_{C/X}$. Therefore, all matrices
\[ M = \begin{pmatrix}
Q & A \\
A^t & L
\end{pmatrix} \]
with any $L$ induce the same $N_{C/X}$. To compute its rank, assume $L$ is invertible, let $I$ be the identity matrix, and use Schur complement:
\[ \begin{pmatrix}
I & -AL^{-1} \\
0 & I
\end{pmatrix}
\begin{pmatrix}
Q & A \\
A^t & L
\end{pmatrix}
\begin{pmatrix}
I & 0\\
-L^{-1}A^t & I
\end{pmatrix} = 
\begin{pmatrix}
Q-AL^{-1}A^t & 0 \\
0 & L
\end{pmatrix}. \]

Since 
\[ \begin{pmatrix}
I & -AL^{-1} \\
0 & I
\end{pmatrix} \ \ \text{ and } \ \ \begin{pmatrix}
I & 0\\
-L^{-1}A^t & I
\end{pmatrix} \]
are invertible, it follows that
\[ \rk \begin{pmatrix}
Q & A \\
A^t & L
\end{pmatrix} = \rk \begin{pmatrix}
Q-AL^{-1}A^t & 0 \\
0 & L
\end{pmatrix} = \rk (Q-AL^{-1}A^t) + \rk (L). \]

And since matrices of rank at least $\rk (Q)$ form an open neighborhood of $Q$, we can choose an invertible matrix $L$ such that $\rk (Q-AL^{-1}A^t)\ge \rk (Q)$, and therefore, $\rk M\ge \rk(Q) + \rk(L)$, that is, the corank of $M$ is at most $\sum _{a_i\ge 4}(a_i-3)$.
\medskip

(b) The polynomial $F$ obtained in Theorem \ref{quadricprop1} is $F = \sum_{j=e+1}^n x_{\gamma_j}x_j$, whose quadratic form matrix is
\[ M = \begin{pmatrix}
0 & A \\
A^t & 0
\end{pmatrix}. \]
The matrix $A$ corresponds to the terms $x_{\gamma_j}x_j$ of $F$. Notice, from the construction of $F$, that $0=\gamma_{e+1}\le \gamma_{e+2}\le \cdots \le \gamma_{n-1}\le \gamma_n=e$, so the rank of $A$ equals the number of different values of $\gamma_j$. Since $b_j = 0\Leftrightarrow \gamma_j = \gamma_{j-1}$, the values of $\gamma_j$ repeat exactly $w$ times. Thus, there are at most $\mathrm{min}\{e+1, n-e-w\}$ different values of $\gamma_j$, so $\rk(A)=\mathrm{min}\{e+1, n-e-w\}$.

As in part (a), we notice that we can add any terms $x_lx_j$ with $l,j\ge e+1$ to $F$ without changing $\psi_F$, thus preserving $N_{C/X}$. So, we can replace $M$ by
\[ M = \begin{pmatrix}
0 & A \\
A^t & L
\end{pmatrix} \]
for any matrix $L$. In particular, let $L=I$ be the identity matrix, and by Schur complement as above, we have
\[ \rk(M) = \rk(0-AIA^t)+\rk (I) = \rk(A) + \rk (I) = \mathrm{min}\{e+1, n-e-w\} + (n-e). \]
Therefore, the corank of $M$ is $e+1 - \mathrm{min}\{e+1, n-e-w\}$.
\end{proof}

Not every possible splitting type, however, comes from a smooth quadric $X$. We can get examples from splitting types of the form $E'$.

\begin{example}
Let $e=3$ and $n=5$. Then there is no smooth quadric $X$ such that $N_{C/X}$ has splitting type $\cO(5)^2\oplus \cO$.
\end{example}
As we have seen in the proof of Theorem \ref{quadricprop1}, any quadric $X$ with this splitting type must be of the form $F = L_4x_4 + L_5x_5$ where $L_4, L_5$ are linear forms in the variables $x_0, \hdots , x_5$. Its quadratic form matrix is a $6\times 6$ matrix of the form
\[ M = \begin{pmatrix}
0 & A \\
A^t & L
\end{pmatrix}. \]
Since $A$ is a $2\times 4$ matrix, $\rk(A)\le 2$, thus, as above, $\rk (M)\le \rk (A) + \rk(L)\le 4$.

\printbibliography[title={References}]

\end{document}